\documentclass[12pt, reqno]{amsart}
\usepackage{amsmath}
\usepackage{amsfonts}
\usepackage{amssymb}
\usepackage[all,cmtip]{xy}           
\usepackage{bm}
\usepackage{bbm}
\usepackage{bbding}
\usepackage{txfonts}
\usepackage{amscd}
\usepackage{xspace}
\usepackage[shortlabels]{enumitem}
\usepackage{ifpdf}

\ifpdf
  \usepackage[colorlinks,final,backref=page,hyperindex]{hyperref}
\else
  \usepackage[colorlinks,final,backref=page,hyperindex,hypertex]{hyperref}
\fi
\usepackage{tikz}
\usepackage[active]{srcltx}

\usepackage{tikz-cd}
\usepackage{tikz}

\usepackage{pdfsync}    

\makeatletter

\newtheorem{thm}{Theorem}[section]
\newtheorem{lem}[thm]{Lemma}

\theoremstyle{definition}

\newtheorem{defi}[thm]{Definition}

\newcommand{\nc}{\newcommand}
\newcommand{\delete}[1]{}

\nc{\mlabel}[1]{\label{#1}}  
\nc{\mcite}[1]{\cite{#1}}  
\nc{\mref}[1]{\ref{#1}}  
\nc{\meqref}[1]{\eqref{#1}}  
\nc{\mbibitem}[1]{\bibitem{#1}} 

\delete{
\nc{\mlabel}[1]{\label{#1}{\hfill \hspace{1cm}{\bf{{\ }\hfill(#1)}}}}
\nc{\mcite}[1]{\cite{#1}{{\bf{{\ }(#1)}}}}  
\nc{\mref}[1]{\ref{#1}{{\bf{{\ }(#1)}}}}  
\nc{\meqref}[1]{\eqref{#1}{{\bf{{\ }(#1)}}}}  
\nc{\mbibitem}[1]{\bibitem[\bf #1]{#1}} 
}

\nc{\tblue}[1]{\textcolor{blue}{#1}}
\nc{\tred}[1]{\textcolor{red}{#1}}
\nc{\tnb}[1]{\textcolor[rgb]{0.00,0.00,0.50}{#1}}
\nc{\tpurple}[1]{\textcolor{purple}{#1}}

\nc{\name}[1]{{\bf #1}}

\nc{\lc}{\lfloor}
\nc{\rc}{\rfloor}
\nc{\dirlim}{\varinjlim}

\nc{\bfk}{\mathbf{K}}
\nc{\invlim}{\displaystyle{\lim_{\longleftarrow}}\,}
\nc{\ot}{\otimes}
\nc{\im}{\mathrm{im}}
\nc{\free}[1]{\overline{#1}}

\nc{\tforall}{\quad \text{for all }}

\nc{\Ad}{\mathrm{Ad}}
\nc{\AD}{\overline{\mathrm{Ad}}}

\nc{\Aut}{\mathrm{Aut}}

\nc{\desc}{descendent\xspace}

\nc{\B}{B}
\newcommand{\g}{\mathfrak g}

\nc{\BA}{{\mathbb A}} \nc{\CC}{{\mathbb C}} \nc{\DD}{{\mathbb D}} \nc{\EE}{{\mathbb E}} \nc{\FF}{{\mathbb F}} \nc{\GG}{{\mathbb G}} \nc{\HH}{{\mathbb H}} \nc{\LL}{{\mathbb L}} \nc{\NN}{{\mathbb N}} \nc{\KK}{{\mathbb K}} \nc{\PP}{{\mathbb P}} \nc{\QQ}{{\mathbb Q}} \nc{\RR}{{\mathbb R}} \nc{\TT}{{\mathbb T}} \nc{\VV}{{\mathbb V}} \nc{\ZZ}{{\mathbb Z}}

\nc{\cala}{{\mathcal A}} \nc{\calc}{{\mathcal C}} \nc{\cald}{{\mathcal D}} \nc{\cale}{{\mathcal E}} \nc{\calf}{{\mathcal F}} \nc{\calg}{{\mathcal G}} \nc{\calh}{{\mathcal H}} \nc{\cali}{{\mathcal I}} \nc{\call}{{\mathcal L}} \nc{\calm}{{\mathcal M}} \nc{\caln}{{\mathcal N}} \nc{\calo}{{\mathcal O}} \nc{\calp}{{\mathcal P}} \nc{\calr}{{\mathcal R}} \nc{\cals}{{\mathcal S}} \nc{\calt}{{\mathcal T}} \nc{\calw}{{\mathcal W}} \nc{\calk}{{\mathcal K}} \nc{\calx}{{\mathcal X}}
\nc{\calz}{{\mathcal Z}}
\nc{\WC}{\mathcal{WC}}

\nc{\bvp}[2]{\boxed{\begin{array}{l}#1\\#2\end{array}}}
\nc{\evl}{E} \nc{\cum}{{\textstyle \varint}} \nc{\mapmonoid}{\frakM}
\nc{\X}{X^{\pm 1}} \nc{\redx}{\text{Red}(X)} \nc{\bre}{\text{bre}}\nc{\dep}{\mathrm{dep}}
\nc{\rbgw}{\frakR(X)} \nc{\rbgo}{\lc \, \rc} \nc{\lbar}{\overline}

\newcommand{\frakx}{\mathfrak x}

\newcommand{\frakM}{\mathfrak M}

\newcommand{\frakR}{\mathfrak R}

\begin{document}

\title[Free operated, differential and Rota-Baxter groups]{Operated groups, differential groups and Rota-Baxter groups with an emphasis on the free objects}

\author{Xing Gao} \address{School of Mathematics and Statistics,
	Key Laboratory of Applied Mathematics and Complex Systems,
	Lanzhou University, Lanzhou, 730000, China}
\email{gaoxing@lzu.edu.cn}

\author{Li Guo}
\address{
Department of Mathematics and Computer Science,
         Rutgers University,
         Newark, NJ 07102}
\email{liguo@rutgers.edu}

\author{Yanjun Liu}
\address{Department of Mathematics, Jiangxi Normal University, Nanchang, China}
\email{liuyanjun@pku.edu.cn}

\author{Zhi-Cheng Zhu}
\address{School of Mathematics and Statistics, Lanzhou University, Lanzhou, Gansu 730000, China}
\email{zhuzhch16@lzu.edu.cn}

\date{\today}

\begin{abstract}
Groups with various types of operators, in particular the recently introduced Rota-Baxter groups, have generated renowned interest with close connections to numerical integrals, Yang-Baxter equation, integrable systems and post-Hopf algebras. This paper gives the general notion of operated groups and provides explicit constructions of free operated groups, free differential groups and free Rota-Baxter groups.
\end{abstract}

\subjclass[2010]
{22E60, 
08B20, 
20D45, 
17B38, 
17B40 
}

\keywords{Operated groups, crossed homomorphism, differential groups, Rota-Baxter group, free object}

\maketitle

\vspace{-.9cm}

\tableofcontents

\vspace{-.9cm}

\allowdisplaybreaks

\section{Introduction}

Linear operators are ubiquitous in mathematics. Notable examples are the endomorphisms and automorphisms in Galois theory, derivations and integral operators in analysis and related areas.
Even though groups do not carry a linear structure, the study of operations on groups has a long history. As noted in ~\mcite{Wio}, groups with endomorphisms were extensively studied by Emmy Noether and her school in the 1920s. In fact, it is this concept that she employed in her original formulation of the three Noetherian isomorphism theorems.
In the classical treatment of groups by  Bourbaki~\cite[pp.~30-31]{Bou}, the notion of groups with endomorphisms was introduced at the very beginning and was then specialized to define modules (in analogy to rings with operators used to define algebras). See also MacLane~\cite[p.~41]{Mac}.
In addition to the endomorphisms, operators on groups with other properties had appeared, such as crossed homomorphisms in group cohomology~\mcite{Se}.

Thanks to the flexibility afforded by the linear structure, linear operators on various algebras have flourished. Along with the ones mentioned above, there were also the difference operators, differential operators {\it with weight}, Rota-Baxter operators, Reynolds operators, averaging operators and Nijenhus operators, as well as their multi-operator analogs,  abstracted from geometry, probability, fluid mechanics, analysis, combinatorics, differential equations and mathematical physics, and with applications in broad areas from quantum field renormalization to mechanical proofs of geometric theorems~\mcite{BGN,Bax,CGM,CK,Da,Gu,Gub,GGL,GK,Kf,Ko,Ku,N,PS,Ri,Ro1,STS,ZtGG,ZGG,ZGG2,ZGM}. 

Putting the corresponding operator identities in a unified framework, Rota raised the question on a classification of operator identities that can be defined on algebras~\mcite{Ro2}. Attempts to address Rota's question have led to a deeper understanding of operator identities in terms of operated algebras, rewriting systems and Gr\"obner-Shirshov bases, and to the discovery of new operator identities~\mcite{GG,GGZh,Guop,GSZ,ZGGS}.

Recently, such enduring enthusiasm on linear operators on algebras has led to renewed interest in operators on groups, beginning with~\mcite{GLS} motivated by the classical Yang-Baxter equation and Poisson geometry~\mcite{Mk,STS}.
More precisely, the notion of Rota-Baxter operators on (Lie) groups was defined in~\mcite{GLS} for which the differentiation gives Rota-Baxter operators on the corresponding Lie algebras. As an application, the fundamental factorization of a Lie group in~\mcite{RS1,RS2,STS} originally obtained indirectly from locally integrating a factorization of a Lie algebra now comes directly from a global factorization of the Lie group, equipped with the Rota-Baxter operator. Incidentally, the formal inverse of a Rota-Baxter operator on a Lie group is none other than the crossed homomorphism on the Lie group, and differentiates to a differential operator on the Lie algebra of left-translation-invariant vector fields.

After this work, there has been a boom of studies on the theory and applications of differential and Rota-Baxter operators on groups.

In~\mcite{BG}, general properties of Rota-Baxter groups, especially extensions of Rota-Baxter groups and Rota-Baxter operators on sporadic simple groups, were studied. In~\mcite{BG2}, it was shown that Rota-Baxter groups give rise to braces and the more general skew left braces in quantum Yang-Baxter equation.
In~\mcite{CS}, a characterization of the gamma functions on a group that come from Rota-Baxter operators as in~\mcite{BG2} was given in terms of the vanishing of a certain element in a suitable second cohomology group.
In~\mcite{CMS}, Rota-Baxter operators on Clifford semigroups were introduced as a useful tool for obtaining dual weak braces.
In~\mcite{BNY,St}, Rota-Baxter groups and skew braces were further studied. 

In~\mcite{JSZ}, a cohomological theory of Rota-Baxter Lie groups was developed which differentiates to the existing cohomologicla theory of Rota-Baxter Lie algebras.
In~\mcite{LS}, a one-to-one correspondence was established between factorizable Poisson Lie groups and quadratic Rota-Baxter Lie groups.
\mcite{Gon} introduced and studied a Rota-Baxter operator on a cocommutative Hopf algebra that generalizes the notions of a Rota-Baxter operator on a group and a Rota-Baxter operator of weight 1 on a Lie algebra.
Finally, in~\mcite{BGST}, post-groups and pre-groups were introduced that can be derived from Rota-Baxter groups and capture the extra structures on (Lie-)Butcher groups in numerical integration, braces, Yang-Baxter equation and post-Hopf algebras~\mcite{BG2,CJO,ESS,LYZ,ML,MW}.

With all these activities on Rota-Baxter groups, differential groups and related structures, it is time to give a systematic study of Rota-Baxter groups and the more general operated groups. As with any algebraic structure, the free objects should play a fundamental role in the study of such group structures. For instance, being able to give an explicit construction of the free objects affords a clear picture on how the operations behave and allows every other object to be expressed as a quotient of a free object.

Our main goal of this paper is the explicit constructions of the free objects, specifically free operated groups, free differential groups and free Rota-Baxter groups.
In Section~\mref{sec:opgp}, after some general discussion of operated groups, free operated groups are constructed.
Then properties and free objects of differential groups are obtained in Section~\mref{sec:difgp}. Finally in Section~\mref{sec:rbgp}, free Rota-Baxter groups are constructed.

\section{Operated groups}
\mlabel{sec:opgp}
In this section we give the notion of operated groups and recall the special cases of Rota-Baxter groups and differential groups. We then construct free operated groups.

\subsection{Notations and examples}
We first give the notion of operated groups. We then provide some examples in both concrete and general terms, some of which will be studied further in later sections.

\begin{defi}
An {\bf operated group} is a pair $(G,P)$ consisting of a group $G$ and a map $P:G\to G$.
A {\bf homomorphism} from an operated group $(G,P)$ to another one $(G',P')$ is a group homomorphism $f:G\to G'$ such that $P'f=fP$.
\end{defi}

Well-known examples of operated groups are groups equipped with group endomorphisms as mentioned in the introduction. In the classical literature such as~\mcite{Bou}, the term groups with  operators means that the operators are group endomorphisms; while for us, the operators do not need to satisfy any condition. For distinction, a group equipped with an group endomorphism will be called an {\bf endo group}. 
Further examples of operated groups are differential groups coming from crossed homomorphisms~\mcite{Se}, and Rota-Baxter groups arising recently from studies of integrable systems and Hopf algebras~\mcite{GLS}.

Recall that a \name{differential operator (of weight $\lambda$)}~\mcite{GK} is an associative algebra $A$ with a linear operator $d:A\to A$ such that
$$ d(uv)=d(u)v+ud(v)+\lambda d(u)d(v) \tforall u, v \in A.$$
The term comes from its classical example as the operator of differential quotient:
$$ d(f)(x):=(f(x+\lambda)-f(x))/\lambda$$
in analysis. Its study in the weight zero case originated from the work of Ritt~\mcite{Ri} on algebraic study of differential equations. Then through the work of Kolchin~\mcite{Ko} and many others (see~\mcite{PS} and the references therein), the study has expended into an area with broad applications. The weighted notion was introduced in~\mcite{GK}.

Similarly, a differential operator (of weight $\lambda$) on a Lie algebra \mcite{LGG,JS,KKL} is a linear map $d:\g\to\g$ such that $$
d[u,v]_\g=[d(u),v]_\g+[u,d(v)]_\g+\lambda[d(u),d(v)]_\g \tforall u,v\in \g.
$$

Let $G$ be a group. For $g\in G,$ define the \name{ adjoint action}
$$\Ad_g:G\longrightarrow G, \quad \Ad_g h:= ghg^{-1} \tforall~ h\in G.$$
Then $\Ad_g$ is in $\Aut(G)$ and the resulting map
$$\Ad: G\to \Aut(G), \quad \Ad(g):=\Ad_g \tforall g\in G,$$
is a group homomorphism.

\begin{defi} \mlabel{de:diffgp}
A map $D$ on a group $G$ is called a \name{ differential operator (of weight $1$)} (resp. $-1$) if
	\begin{equation}
		\mlabel{eq:diffgp}
		D(gh)=D(g)\Ad_gD(h) \quad  \Big(\text{resp. } D(gh)=(\Ad_gD(h))D(g)\Big) \tforall g, h\in G.
	\end{equation}
Then the pair $(G,D)$ is called a \name{differential group of weight $1$} (resp. $-1$).
\end{defi}
It remains an open problem to find a suitable notion of differential operator of weight zero on a group. 

\emph{For the rest of the paper, we only work with differential group of weight $1$, which we simply call a {\bf differential group}.}

As shown in~\mcite{GLS}, the derivation of a differential operator on a Lie group is a differential operator (of weight 1) on the corresponding Lie algebra.
This notion is the special case of the classical notion of a crossed homomorphism or a 1-cocycle with non-abelian coefficients~\cite{Se}.

\begin{defi}
	Let $G$ be a group and let $\Gamma$ be a group with a group action by $G$, defined by $\alpha\mapsto \alpha^x, \alpha\in \Gamma, x\in G$. A map $f: G\to \Gamma$ is called a {\bf crossed homomorphism} or a {\bf 1-cocycle} if
	\begin{equation} \notag
		f(xy)=f(x)f(y)^x  \tforall x, y \in G.
	\end{equation}
\end{defi}

Crossed homomorphism is the key notion for group cohomology and Galois cohomology, especially for noncommutative groups.
Taking $\Gamma$ to be the group $G$ itself equipped with the adjoint action of $G$, we obtain the notion of the differential operator on groups. In general, a crossed homomorphism can be interpreted as a relative differential operator on a group associated with a group representation~\mcite{BGST}.

The notion of a Rota-Baxter group is obtained in~\mcite{GLS} in an effort to find the integration of a Rota-Baxter operator on a Lie algebra in order to study the factorization on a Lie group without having to integrate from a factorization on the Lie algebra~\mcite{RS1,RS2,STS}.
\begin{defi} \mlabel{de:rbg}
A \name{ Rota-Baxter group} (of weight 1) is a group $G$ with a map $\B:G\longrightarrow G$ satisfying the following equation
\begin{eqnarray}\mlabel{RBgroup}
		\B(g_1)\B(g_2)=\B(g_1\Ad_{\B(g_1)} g_2) \tforall~ g_1,g_2\in G,
\end{eqnarray}
called the \name{ Rota-Baxter relation} for  groups.
\end{defi}

When $G$ is an abelian group, Eq.~\eqref{RBgroup} means that $\B$ is a group homomorphism.

A Rota-Baxter operator of weight 1 on a group is also the right inverse of a differential operator of weight 1, as in the case of operators on associative algebras and Lie algebras.

	Let $(G,\B)$ and $(G',\B')$ be Rota-Baxter groups. A map $\Phi:G\longrightarrow G'$ is a \name{Rota-Baxter group homomorphism} if $\Phi$ is a group homomorphism such that $\Phi \B=\B' \Phi.$

\begin{defi}\mlabel{de:wt-1}
We define a \name{ Rota-Baxter group of weight $-1$} to be a group $G$ with a map $C:G\longrightarrow G$ such that
\begin{eqnarray}\mlabel{RBgroup-1} \notag
C(g_1)C(g_2)=C((\Ad_{C(g_1)} g_2)g_1) \tforall~ g_1,g_2\in G.
\end{eqnarray}
\end{defi}

It still unknown how to give a suitable notion of Rota-Baxter operator of weight zero on a group such that the derivation is the Rota-Baxter operator of weight zero on a Lie algebra. 

Let $\B$ be a Rota-Baxter operator (of weight 1) on a group $G$. Define $C:G\to G$ by
$$ C(g):=\B(g^{-1}).$$
In Eq.~\meqref{RBgroup}, replacing $g_1$ and $g_2$ by $g_1^{-1}$ and $g_2^{-1}$, we obtain
$$\B(g_1^{-1})\B(g_2^{-1}) =\B(g_1^{-1}\Ad_{\B(g_1^{-1})}g_2^{-1})
=\B\Big(\Big(\big(\Ad_{\B(g_1^{-1})}g_2\big) g_1 \Big)^{-1}\Big).$$
This gives
$$C(g_1)C(g_2)=C\Big(\big(\Ad_{C(g_1)}g_2\big)g_1\Big).
$$
Therefore, $C$ is a Rota-Baxter operator of weight $-1$.

Let $G_1$ and $G_2$ be subgroups of a group $G$ such that $G=G_1G_2$ and $G_1\cap G_2=\{1\}$. Note that it is not necessary that $G$ is the direct product of $G_1$ and $G_2$.  Let $C$ be the \name{projection to the first direct factor}, i.e.
$$
C(g_+g_-):=g_+.
$$
Then it is straightforward to deduce that $C$ is a Rota-Baxter operator of weight $-1$~\mcite{GLS}.

{\it For the rest of the paper, by a {\bf Rota-Baxter group} we mean a Rota-Baxter group of weight $1$.}

\subsection{Free operated groups}
\mlabel{ss:freeopgp}

The notion of free operated groups can be defined in the language of universal algebra by a universal property. Here we give an explicit construction of a free operated group on a set by bracketed words.

\begin{defi}
	Let $X$ be a set. The \name{free operated group} generated by $X$ is an operated group $(F(X),P_X)$ together with a map $j_X:X\to F(X)$ with the property that, for each operated group $(G,P)$ and map $f:X\to G$, there is a unique homomorphism $\free{f}:(F(X),P_X)\to (G,P)$ of operated groups such that $\free{f} j_X=f$. 	
\end{defi}

We first give some general notations. For any set $Y$, let $F(Y)$ denote the free group generated by $Y$, via the usual construction of reduced words in $Y\sqcup Y^{-1}$. So $F(Y)$ consists of the identity $1$ and reduced words in the letter set $Y\cup Y^{-1}$, where a reduced word is an element of the form
$$ w=w_1\cdots w_k,$$
with $w_i\in Y\sqcup Y^{-1}$ so that no adjacent letters $w_i$ and $w_{i+1}$ are of the form $z z^{-1}$ for a $z\in Y\sqcup Y^{-1}$.
Also let $\lc Y \rc$ denote a set in bijection with $Y$ but disjoint from $Y$. Elements in $\lc Y\rc$ are denoted by $\lc y\rc, y\in Y$.

Now let $X$ be a given set. We define a directed system of free groups $\calg_n:=\calg_n(X), n\geq 0,$ by a recursion. First define the free group
$$ \calg_0: = F(X).$$
Then define the free group
$$\calg_1:=F(X\sqcup \lc F(X)\rc)$$
generated by $X\sqcup \lc F(X)\rc$ and let
$$ i_{0,1}: \calg_0\to \calg_1$$
denote the group homomorphism of free groups induced by the natural inclusion $X\to X\sqcup \lc F(X)\rc$.

Recursively, for any given integer $n\geq 1$, suppose that
the groups $\calg_i, i\leq n,$ have been constructed and the injective group homomorphisms
$$i_{i-1, i}: \calg_{i-1}\to \calg_i, i\leq n,$$
have been defined. Then define the free group
$$ \calg_{n+1}:=F(X\sqcup \lc \calg_n\rc)$$
generated by the set $X\sqcup \lc \calg_n\rc$.

Furthermore, from the injective group homomorphism $i_{n-1,n}:\calg_{n-1}\to \calg_n$, we obtain the injections \begin{equation}
	\lc \calg_{n-1}\rc \to \lc \calg_n\rc, \quad \lc w\rc \mapsto \lc i_{n-1,n}(w)\rc \tforall w\in \calg_{n-1}
	\mlabel{eq:btrans}
\end{equation} and $X\sqcup \lc \calg_{n-1}\rc \to X\sqcup \lc \calg_n\rc$. The latter then induces the injective homomorphism of free groups
\begin{equation} \notag
	i_{n,n+1}: \calg_n=F(X\sqcup \lc \calg_{n-1}\rc)  \to \calg_{n+1}=F(X\sqcup \lc \calg_n\rc).
	\mlabel{eq:opgpincl}
\end{equation}
This completes the recursive construction of the directed system $\{\calg_n\}_{n\geq 0}$ of free groups. Finally, we define the direct limit of the directed system

$$ \calg:=\calg(X):=\dirlim \calg_n = \cup_{n\geq 0} \calg_n.$$
It is still a group which naturally contains $\calg_n, n\geq 0,$ as subgroups via the structural injective group homomorphisms
\begin{equation}
	i_n:\calg_n\to \calg.
	\mlabel{eq:nind}
\end{equation}

Denote 
$$X^{\pm 1}:= X \sqcup X^{-1},\,\lc \calg\rc^{-1}:= \{w^{-1} \mid w\in \lc \calg\rc\}\,\text{ and }\,\lc \calg\rc^{\pm 1}:= \lc \calg\rc \sqcup \lc \calg\rc^{-1}. $$
For a $w\in \calg$, we may uniquely write
\begin{equation}
\mlabel{eq:frbg}
w = w_1\cdots w_k\, \text{ with }\, w_i\in X^{\pm 1} \sqcup \lc \calg\rc^{\pm 1},
\end{equation}
with no adjacent factors being the inverse of each other.
The factorization is called the \name{standard factorization} or \name{standard form} of $w$.
Define the \name{breadth} of $w$ to be $\bre(w):=k$ and \name{depth} of $w$ to be
$$\dep(w):= \max \{ \dep(w_i) \mid 1\leq i\leq k\}, \,\text{ where }\, \dep(w_i) := \min\{n \mid w_i\in \calg_n\setminus \calg_{n-1} \}.$$

We next define an operator $P_X$ on $\calg$. For $u\in \calg$, we have $u\in \calg_n$ for some $n\geq 0$. We then define $P_X(u)=\lc u\rc$ which is in $\calg_{n+1}$ and hence in $\calg$. It is direct to check that $P_X(u)$ is independent of the choice of $n$ such that $u$ is in $\calg_n$. Thus the pair $(\calg(X),P_X)$ is an operated group. Finally let
$$ j_X: X\to \calg$$
be the natural inclusion.

\begin{thm} \mlabel{thm:freeopgp}
	The operated group $(\calg(X),P_X)$ together with the map $j_X:X\to \calg(X)$ is the free operated group on $X$.
\end{thm}

\begin{proof}
We just need to verify that the operated group $(\calg(X),P_X)$ together with $j_X$ satisfies the desired universal property of a free operated group on $X$. For this purpose, let $(G,P)$ be a given operated group and $f:X\to G$ be a map. We will apply a recursion to define a sequence of group homomorphisms
$$ f_n: \calg_n \to G, n\geq 0, $$
	that yield a group homomorphism $\free{f}: \calg\to G$ by taking the direct limit.
	
	First define $f_0:\calg_0=F(X)\to G$ by the universal property of the free group $F(X)$ on $X$ such that $f_0j_X=f$. We then define a map
	$$\hat{f}_0:\lc \calg_0\rc \to G, \quad
	\hat{f}_0(\lc w\rc):= P(f_0(w))\quad \text{for all } w\in \calg_0.$$
	We then define a group homomorphism
	$f_1: \calg_1:=F(X\sqcup \lc \calg_0\rc) \to G$
	by the universal property of the free group $\calg_1$ induced by the set map
	$$X\sqcup \lc \calg_0\rc \to G,\quad x\mapsto f(x), \lc w\rc \mapsto \hat{f}_0(\lc w\rc ) \tforall x\in X, \lc w\rc\in \lc \calg_0\rc.$$
	Then we have
	$f_1 i_{0,1}=f_0$.
	
Recursively, for any given integer $n\geq 1$, suppose that, for $i\leq n$, the group homomorphisms $f_i:\calg_i\to G$ have been defined such that $f_i i_{i-1,i}=f_{i-1}$. We then define
	$$\hat{f}_n: \lc \calg_n\rc \to G, \quad \lc w\rc \mapsto P(f_n(w)) \tforall w\in \calg_n,$$
	and then extend it to a map
	$$ X\sqcup \lc \calg_n\rc \to G$$
	by sending $x\in X$ to $f(x)$. We then apply the universal property of the free group $\calg_{n+1}:=F(X\sqcup \lc \calg_n\rc)$ and obtain the group homomorphism
	$$ f_{n+1}: \calg_{n+1} \to G.$$
	Furthermore, to show that $f_{n+1}i_{n,n+1}=f_n$, we just need to check
	\begin{equation} f_{n+1}i_{n,n+1}(w)=f_n(w) \tforall w\in X\sqcup \lc \calg_{n-1} \rc.
		\mlabel{eq:frest}
	\end{equation}
This is clear for $w\in X$ since both sides equal to $f(x)$. For $w\in \lc \calg_{n-1}\rc$, we write $w=\lc \lbar{w}\rc$ for $\lbar{w}\in \calg_{n-1}$. Then by Eq.~\meqref{eq:btrans} we obtain
$$i_{n,n+1}(w)=i_{n,n+1}(\lc \lbar{w}\rc) = \lc i_{n-1,n}(w)\rc$$
and hence by Eq.~\meqref{eq:btrans},
$$f_{n+1}i_{n,n+1}(w)=f_{n+1}(\lc i_{n-1,n}(\lbar{w})\rc)
= \hat{f}_n(\lc i_{n-1,n}(\lbar{w})\rc)$$
$$=P(f_ni_{n-1,n}(\lbar{w}))
=P(f_{n-1}(\lbar{w})) = f_n(\lc \lbar{w}\rc) =f_n(w).$$
This proves Eq.~\meqref{eq:frest}, as needed.
	
Thus we can take the direct limit of $f_n, n\geq 1,$ and obtain a group homomorphism
$$\free{f}=\dirlim f_n: \calg \to G.$$
	By construction, we have $\free{f} j_X =f$. The relation
	$$ f_{n+1}(\lc w\rc)= \lc f_n(w)\rc \tforall w\in \calg_n,$$
	gives $\free{f}P_X=P\free{f}$. Therefore $\free{f}:(\calg(X),P_X) \to (G,P)$ is a homomorphism of operated groups.
	
Finally, by the construction of $\free{f}$, it is the unique homomorphism of operated group homomorphism such that $\free{f}j_X=f$. More precisely, let $\free{f}':(\calg,P_X)\to (G,P)$ be another homomorphism of operated algebra homomorphisms such that $\free{f}' j_X=f$. Then we recursively verify that $\free{f} i_n = \free{f}' i_n, n\geq 0,$ for the group homomorphisms $i_n:\calg_n\to\calg$ in Eq.~\meqref{eq:nind}. This gives $\free{f}'=\free{f}$. Therefore the desired universal property is verified.
\end{proof}

Free operated semigroups have been constructed in~\mcite{Guop} in terms of bracketed words from $X$. One can think of elements of $\calg(X)$ as reduced bracketed words from $X\cup X^{-1}$. Here a bracketed word is reduced if it does not have adjacent pair of subwords that are the inverses of each other. It would be interesting to obtain combinatorial construction of free operated groups in terms of path or trees as in the case of operated semigroups~\mcite{Guop}, or in terms of Cayley graphs for free groups, which are fundamental in geometric group theory~\mcite{BFH1,BFH2}. 

\section{Free differential groups}
\mlabel{sec:difgp}

With the notion of a differential group given in Definition~\mref{de:diffgp}, we now provide some properties and then construct free differential groups.

\begin{lem} \mlabel{lem:asso}
Let $(G,d)$ be a differential group.
\begin{enumerate}
	\item
\mlabel{it:diffassoc}
The derivation $d$ is compatible with the associativity of $G$. More precisely, for $g,h,k\in G$, the results of $d((gh)k)$ and $d(g(hk))$ agree. In fact, we have
\begin{equation}\mlabel{eq:asso}
d(ghk)=d(g)gd(h)hd(k)h^{-1}g^{-1} \quad \text{for all } g, h, k\in G.
\end{equation}
\item
\mlabel{it:prod}
More generally, for $g_1,\ldots,g_n\in G, n\geq 2$, we have
\begin{equation} \notag
 \mlabel{eq:prod}
 d(g_1\cdots g_n)=\Big(\prod_{i=1}^n (d(g_i)g_i)\Big) \big(g_1\cdots g_n\big)^{-1} \quad \text{for all }g_1,\ldots,g_n\in G.
\end{equation}
\end{enumerate}
\end{lem}
\begin{proof}
\meqref{it:diffassoc}
We directly check that
$$ d((gh)k)=d(gh)ghd(k)(gh)^{-1}=d(g)gd(h)g^{-1}ghd(k)(gh)^{-1}
=d(g)gd(h)hd(k)h^{-1}k^{-1}.
$$
Also
$$ d(g(hk))=d(g)gd(hk)g^{-1}=d(g)gd(h)hd(k)h^{-1}g^{-1}.
$$
Hence Eq.~\meqref{eq:asso} holds.

\smallskip

\noindent
\meqref{it:prod}
The general equation follows from a simple induction on $n$.
\end{proof}

\begin{lem}
\mlabel{lem:unitinv}
Let $(G,d)$ be a differential group. Then we have
\begin{equation}
\mlabel{eq:unitinv}
d(1)=1, \quad d(g^{-n})=\big(g^{-1} d(g)^{-1}\big)^n g^n \tforall g\in G, n\geq 1.
\end{equation}
\end{lem}

\begin{proof}
Denote $d(1)=a$. Then from
$$ a=d(1)=d(1^2)=d(1)d(1)=a^2,$$
we obtain $a=1$.
To prove the second equality, we apply induction on $n\geq 1$. Let $n=1$. Then for $g\in G$, we have
$$ 1=d(1)=d(g g^{-1})=d(g)gd(g^{-1})g^{-1},$$
which yields
$$ d(g^{-1})=g^{-1}d(g)^{-1}g. $$
For $k\geq 1$, assume
$$ d(g^{-k})=(g^{-1}d(g)^{-1})^kg^k.$$
Then we obtain
\begin{eqnarray*}
	d(g^{-(k+1)}) &=& d(g^{-1} g^{-k})\\
	&=& d(g^{-1})g^{-1} d(g^{-k})g\\
	&=& g^{-1}d(g)^{-1}gg^{-1} \big(g^{-1}d(g)^{-1}\big)^kg^kg\\
	&=& g^{-1}d(g)^{-1}\big(g^{-1}d(g)^{-1}\big)^k g^{k+1}\\
	&=& \big(g^{-1}d(g)^{-1}\big)^{k+1}g^{k+1}.
\end{eqnarray*}
This completes the induction.
\end{proof}

We now construct free differential groups by adding a differential structure to the classical construction of free groups.

\begin{defi}
The {\bf free differential group} on a set $X$ is a differential group $(F\{X\},D)$ together with a set map $i=i_X:X\to F\{X\}$ with the property that,
for any differential group $(G,d)$ and set map $f:X\to G$, there is a unique homomorphism of differential groups $\free{f}:(F\{X\},D_X) \to (G,d)$ such that $\free{f}i =f$.
\mlabel{de:freediffgp}
\end{defi}

As a side remark, the notion of a free endo group can be given in the same way and can be constructed in the same way as below. 

Let $X$ be a set. Denote
$$ \Delta X:= X\times \NN = \{x^{(n)}:=(x,n)\,|\, x\in X, n\in\NN\}.$$
Let
$$F\{X\}:=F(\Delta X)$$
denote the free group on the set $\Delta X$, as usual in terms of reduced words in the alphabet $\Delta X$ and their formal inverses.
We define an operator $D:=D_X$ on $F\{X\}$ in two steps. First define
$$D:\Delta X \to \Delta X, \quad x^{(n)}\mapsto x^{(n+1)}, x^{(n)}\in \Delta X=X\times \NN.$$
Then extend $D$ to $F\{X\}=F(\Delta X)$ by applying Eq.~\meqref{eq:diffgp}. More precisely, an element $\frakx$ of $F(\Delta X)$ can be uniquely expressed as a reduced word
$$\frakx:=\frakx_1\cdots \frakx_k, \quad \frakx_i\in \Delta X\cup (\Delta X)^{-1},$$
with the convention that the length of the empty word $\frakx=1$ is $0$.
When $k=0$, we define $D(1)=1$. When $k=1$, we define \begin{equation} \notag
D(\frakx):=D(x^{(n)})=x^{(n+1)}, \quad D(\frakx^{-1}):=D((x^{(n)})^{-1})=(x^{(n)})^{-1}(x^{(n+1)})^{-1} x^{(n)}, \quad n\geq 0.
\mlabel{eq:derpower}
\end{equation}
For a given integer $m\geq 1$, assume that $D(\frakx)$ has been defined for all reduced words $\frakx$ of length $k=m$. Then for a reduced word $\frakx=\frakx_1\cdots \frakx_{m+1}$ of length $m+1$, we recursively define

\begin{equation} \mlabel{eq:freed}
D(\frakx)=D(\frakx_1)\frakx_1 D(\frakx_2\cdots \frakx_{m+1}) \frakx_1^{-1}.
\end{equation}

Let
$$i_X:X\to \Delta X \to F\{X\}, \quad x\mapsto x^{(0)}, x\in X,$$
be the canonical inclusion.

\begin{thm} \mlabel{thm:freedgp}
Let $X$ be a set.
 \begin{enumerate}
\item
The pair $(F\{X\},D_X)$ is a differential group.
\mlabel{it:freedgp1}
\item
The triple $(F\{X\},D_X,i_X)$ is a free differential group on $X$.
\mlabel{it:freedgp2}
\end{enumerate}
\end{thm}

\begin{proof}
\eqref{it:freedgp1}
We just need to verify that $D$ is a differential operator on the group $F\{X\}$, that is, Eq.~\meqref{eq:diffgp} holds for all $g, h\in F\{X\}$. This can be done by an induction on the sum $m+n$ of the lengths $m$ and $n$ of $g$ and $h$ as follows.

First consider $m+n=0$. Then $g=h=1$ and so
$$ D(gh)=D(1)=1= D(g)gD(h)g^{-1}.$$
Thus Eq.~\meqref{eq:diffgp} holds. Let $k\geq 0$ be given. Assume that Eq.~\meqref{eq:diffgp} holds for all $g, h\in F\{X\}$ with $m+n=k$ and consider $g, h\in F\{X\}$ with $m+n=k+1$. So either $m>0$ or $n>0$. First suppose $m>0$. Then $g=\frakx_1 g_1$ for $\frakx_1\in \Delta X, g_1\in F\{X\}$. Then
$$ D(gh)=D(\frakx_1g_1h)=D(\frakx_1) \frakx_1 D(g_1h)\frakx_1^{-1}.$$
Also,
$$ D(g)gD(h)g^{-1}=D(\frakx_1)\frakx_1D(g_1)\frakx_1^{-1} gD(h)g^{-1}
= D(\frakx_1)\frakx_1 D(g_1) g_1D(h) g_1^{-1} \frakx_1^{-1}
= D(\frakx_1)\frakx_1 D(g_1h)\frakx_1^{-1},$$
where the last step utilizes the induction hypothesis. The case of $n>0$ follows from the same argument. This completes the induction.

\smallskip
\noindent
\meqref{it:freedgp2}
We just need to verify that the differential group $(F\{X\},D)$ together with $i_X:X\to F\{X\}$ satisfies the expected universal property stated in the theorem.
Now let a differential group $(G,d)$ and a map $f:X\to G$ be given. Define a differential group homomorphism $\free{f}:F\{X\}\to G$ by first taking
$$ \hat{f}:\Delta X\to G, \quad x^{(n)}\mapsto d^n(f(x)) \tforall x^{(n)}\in \Delta X,$$
and then taking
$$ \free{f}: F\{X\}=F(\Delta X) \to G$$
to be the unique group homomorphism from $\hat{f}$ by the universal property of the free group $F(\Delta X)$ on $\Delta X$.

We are left to verify that $\free{f}$ is a homomorphism of differential groups satisfying $\free{f}i=f$ and is the unique one with this property. For the first property, we just need to check that $\free{f}$ is compatible with the differential operators: $\free{f} D=d\free{f}$ which is proved by an induction on the length by applying Eq.~\meqref{eq:freed}. Indeed, it is direct to check that  $\free{f}D(g)=d\free{f}(g)$ if $g=1$ or if $g$ is in $\Delta X$.  Further if $g=\frakx_1g_1$ with $\frakx_1\in \Delta X$, then applying the induction hypothesis to $g_1$, we obtain
\begin{eqnarray*}
\free{f}D(\frakx)&=&\free{f}D(\frakx_1g_1)
= \free{f}\big(D(\frakx_1)\frakx_1 D(g_1)\frakx_1^{-1}\big)
= \free{f}D(\frakx_1)\free{f}(\frakx_1)\free{f}D(g_1)\free{f}(\frakx_1^{-1})\\
&=&D(\free{f}(\frakx_1)) \free{f}(\frakx_1)D(\free{f}(g_1))\free{f}(\frakx_1)^{-1}
=D\big(\free{f}(\frakx_1) \free{f}(g_1)\big)
=D\free{f}(g).
\end{eqnarray*}
This completes the induction.

The construction of $\free{f}$ shows that it is the only way to obtain a differential group homomorphism, showing the uniqueness of $\free{f}$.
\end{proof}

\section{Free Rota-Baxter groups}
\mlabel{sec:rbgp}

We now construct the free Rota-Baxter group generated by a set. To be specific, we first give the definition.

\begin{defi}
Let $X$ be a set. The \name{free Rota-Baxter group} generated by $X$ is a Rota-Baxter group $(\frakR(X),B_X)$ together with a map $j_X:X\to \frakR(X)$ such that, for any Rota-Baxter group $(G,B)$ and map $f:X\to G$, there is a unique homomorphism $\free{f}:(\frakR(X),B_X)\to (G,B)$ of Rota-Baxter groups such that $\free{f}j_X=f$.
\end{defi}

From the construction of free operated groups obtained in Section~\mref{ss:freeopgp}, the free Rota-Baxter group can be obtained by taking a quotient as follows.

Let a set $X$ be given and let $(\calg(X),P_X)$ be the free operated group generated by $X$ constructed in Theorem~\mref{thm:freeopgp}. Let $N\coloneqq N_X$ be the normal operated subgroup of $\calg(X)$ generated by elements of the form
$$ P_X(u)P_X(v) \Big(P_X(u\Ad_{P_X(u)}v)\Big)^{-1}
\tforall u, v\in \calg(X).$$
Then the quotient $\calg(X)/N$, with its operator induced from $P_X$ modulo $N$, is automatically a free Rota-Baxter group generated by $X$. We next give a direct and explicit construction of the free Rota-Baxter group on $X$, by displaying a canonical subset of $\calg(X)$ and equipping it with the desired Rota-Baxter group structure. The construction is inspired by the construction of free Rota-Baxter (associative) algebras~\mcite{Gub}.

\begin{defi}
An element $w\in\calg(X)$ is called a \name{Rota-Baxter group word} on $X$ if $w$ contains no subword $\lc u\rc\lc v\rc$ with $u, v\in \calg(X)$.
Denote by $\rbgw$ the set of all Rota-Baxter group word on $X$.
\end{defi}

In terms of the standard factorization in Eq.~\meqref{eq:frbg}, each Rota-Baxter group word may be uniquely written as
\begin{equation}
	w=w_{1}\cdots w_{k}
	\mlabel{eq:stform}
\end{equation}
where $w_{i}\in X^{\pm 1}\sqcup\lc \calg(X)\rc^{\pm 1}$ and no adjacent $w_i$ and $w_{i+1}$ are both in $\lc \calg(X)\rc$ or both in $\lc \calg(X)\rc^{-1}$. Note that it is okay if one of $w_i$ or $w_{i+1}$ is in $\lc \calg(X)\rc$ and the other one in $\lc \calg(X)\rc^{-1}$.
We call Eq.~\eqref{eq:stform} the {\bf standard factorization} of $w$. 

We are going to define a multiplication $\diamond$ on $\rbgw$. First we define $1$ to be the identity: $1\diamond u=u\diamond 1=u$ for all $u\in \rbgw$. For $u, v\in \rbgw\backslash \{1\}$, we apply induction on $(\dep(u), \dep(v))\geq (0,0)$ lexicographically.

Let $m,n\geq 0$ be given such that $(m,n)>(0,0)$ lexicographically. Suppose that the product $u\diamond v$ have been defined for $u, v\in \rbgw$ with $(\dep(u), \dep(v))<(m,n)$ lexicographically and consider $u, v\in \rbgw$ with $(\dep(u), \dep(v))=(m,n)$.
We have the following four cases to check, with the first three cases dealing with $\bre(u)=\bre(v)=1$.

\noindent{\bf Case 1.} Let $u=\lc \lbar{u}\rc$ and $v=\lc \lbar{v}\rc$ for some $\lbar{u}$ and $\lbar{v}\in \rbgw$.
Write $\lbar{v} = \lbar{v}_1\cdots \lbar{v}_k$ in the standard factorization. Then the expression
\begin{align}
	\AD_u \lbar{v} :=
	\begin{cases}
		(u \diamond \lbar{v}_1) \lbar{v}_2 \cdots \lbar{v}_{k-1} (u\diamond \lbar{v}'_k)^{-1} , & \text{ if } \lbar{v}_k = (\lbar{v}'_k)^{-1} \in \lc \rbgw\rc^{-1}, \\
		(u \diamond \lbar{v}) u^{-1}, & \mbox{otherwise}, \\
	\end{cases}
	\mlabel{eq:diam}
\end{align}
is well defined by the induction hypothesis on $(\dep(u),\dep(v))$. In particular if $k=1$, then
\begin{align} \notag
	\AD_u \lbar{v} =
	\begin{cases}
		u (u\diamond \lbar{v}')^{-1} , & \text{ if } \lbar{v} = (\lbar{v}')^{-1} \in \lc \rbgw\rc^{-1}, \\
		(u \diamond \lbar{v})u^{-1}, & \mbox{otherwise}. \\
	\end{cases}
\end{align}
Then we define
\begin{align}
u\diamond v := \lc\free{u} \diamond \AD_{u} \lbar{v} \rc.
\mlabel{eq:dia}
\end{align}

\noindent{\bf Case 2.}  Let $u= \lc \lbar{u}\rc^{-1}$ and $v=\lc \lbar{v}\rc^{-1}$ for some $\lbar{u}, \lbar{v}\in \rbgw$.
Then define
\begin{equation}
u\diamond v := (\lc \lbar{v}\rc \diamond \lc \lbar{u}\rc)^{-1},
\mlabel{eq:uvconc}
\end{equation}
where $\lc \lbar{v}\rc \diamond \lc \lbar{u}\rc$ is reduced to Case 1.

\noindent{\bf Case 3.} Let $u$ and $v$ have breadth 1, but be not in the above two cases. Then the concatenation $uv$ is still in $\rbgw$. Then we can define
\begin{equation}
u\diamond v:= uv.
\mlabel{eq:conc}
\end{equation}

\noindent
{\bf Case 4.}
In general, write $u=u_{1}\cdots u_{k}$ and $v=v_{1}\cdots v_{\ell}$ with $k,\ell\geq 1$ in the standard forms. Define
\begin{align}
u\diamond v := u_{1} \cdots u_{k-1} (u_{k}\diamond v_{1}) v_{2}\cdots v_{\ell},
\mlabel{eq:stdia}
\end{align}
where $u_{k}\diamond v_{1}$ is defined by Cases 1, 2, 3.
Notice that $1\diamond u = u=u\diamond 1 $ for $u\in \calg$.

Note that by the construction of $\rbgw$, the set $\rbgw$ is closed under taking the bracket: for $w\in \rbgw$, we have $\lc w \rc\in \rbgw$.

\begin{thm}
With the above definition of $\diamond$, the triple $(\rbgw,\diamond,\rbgo)$ is a Rota-Baxter group of weight 1.
\mlabel{thm:rtgp1}
\end{thm}

\begin{proof}
By the construction of $\rbgw$, we have $1\in \rbgw$ and $w^{-1}\in \rbgw$ provided $w\in \rbgw$. So we are left with checking the associativity of $\diamond$ and the Rota-Baxter relation in Eq.~\meqref{RBgroup}. 
 first verify the associativity of $\diamond$:
 \begin{equation}
(u\diamond v)\diamond w=u\diamond(v\diamond w) \,\text{ with }\, u, v, w\in \rbgw.
\mlabel{eq:as}
 \end{equation}
which will be accomplished by an induction on $\bre(u)+\bre(v)+\bre(w)\geq 3$.

\smallskip

\noindent
{\bf I. The initial step on $\bre(u)+\bre(v)+\bre(w) $.}
For the initial step of $\bre(u)+\bre(v)+\bre(w)=3$, we have $\bre(u)=\bre(v)=\bre(w)=1$. In this case, we apply the induction on $(\dep(u), \dep(v), \dep(w))\geq(0, 0, 0)$ lexicographically. If
$(\dep(u), \dep(v), \dep(w))= (0, 0, 0)$, then $u,v,w$ are in the free group $\calg_0$ and Eq.~(\mref{eq:as}) holds by the associativity of the free group $\calg_0$.

For the inductive step, let $p,q,r \geq 0$ be given such that $(p,q,r)>(0,0,0)$ lexicographically.
Assume that Eq.~(\ref{eq:as}) has been proved for $u, v, w\in \rbgw$ with $\bre(u)=\bre(v)=\bre(w)=1$ and $(\dep(u), \dep(v), \dep(w))<(p, q, r)$ lexicographically, and consider $u, v, w\in \rbgw$ with $\bre(u)=\bre(v)=\bre(w)=1$ and $(\dep(u), \dep(v), \dep(w))=(p, q, r)$.
Then according to whether all of $u, v, w$ are in $\lc \rbgw\rc$ or not, we consider the following four cases.

\noindent{\bf Case 1.} Suppose that all $u, v, w$ are in $\lc \rbgw\rc$. Then we take $u=\lc\lbar{u}\rc, v=\lc \lbar{v}\rc, w=\lc \lbar{w}\rc$ for some $\lbar{u}, \lbar{v}, \lbar{w} \in \rbgw$. Write $\lbar{v}=\lbar{v}_{1}\cdots \lbar{v}_{m}$ and $\lbar{w}=\lbar{w}_{1}\cdots \lbar{w_{n}}$ in the standard forms. We divide this case to four subcases according to whether $\lbar{v}_{m}$ and $\lbar{w}_{n}$ are in $\lc \rbgw\rc^{-1}$ or not.

\noindent{\bf Case 1.1.}  Suppose $\lbar{v}_{m}=(\lbar{v}_{m}')^{-1}, \lbar{w}_{n}=(\lbar{w}_{n}')^{-1}\in \lc \rbgw\rc^{-1}$ for some $\lbar{v}_{m}', \lbar{w}_{n}'\in \lc \rbgw\rc$. Then on the one hand,
\begin{align*}
&\ (u\diamond v)\diamond w \\
=&\ (\lc \lbar{u} \rc\diamond \lc \lbar{v} \rc)\diamond \lc \lbar{w} \rc \\
=&\ \Big\lc \lbar{u}\diamond \big((\lc \lbar{u} \rc\diamond \lbar{v}_{1})\lbar{v}_{2}\cdots \lbar{v}_{m-1}(\lc \lbar{u} \rc\diamond \lbar{v}_{m}')^{-1}\big)\Big\rc \diamond \lc \lbar{w} \rc \quad(\text{Eqs.~(\ref{eq:dia}) and~(\ref{eq:diam})})\\
=&\ \bigg\lc\Big(\lbar{u}\diamond \big((\lc \lbar{u} \rc\diamond \lbar{v}_{1})\lbar{v}_{2}\cdots \lbar{v}_{m-1}(\lc \lbar{u} \rc\diamond \lbar{v}_{m}')^{-1}\big)\Big)\diamond \Big(\big\lc \lbar{u} \diamond (\lc \lbar{u} \rc\diamond \lbar{v}_{1})\lbar{v}_{2}\cdots \lbar{v}_{m-1}(\lc \lbar{u} \rc\diamond \lbar{v}_{m}')^{-1}\big\rc\diamond \lbar{w}_{1}\Big) \\
&\ \lbar{w}_{2}\cdots \lbar{w}_{n-1}\big(\big\lc \lbar{u}\diamond (\lc \lbar{u} \rc\diamond \lbar{v}_1)\lbar{v}_{2}\cdots \lbar{v}_{m-1}(\lc \lbar{u} \rc\diamond \lbar{v}_{m}')^{-1}\big\rc\diamond \lbar{w}_{n}'\big)^{-1}\bigg\rc  \quad\quad(\text{Eqs.~(\ref{eq:dia}) and~(\ref{eq:diam})})\\
=&\ \bigg\lc\Big(\lbar{u}\diamond \big((\lc \lbar{u} \rc\diamond \lbar{v}_{1})\lbar{v}_{2}\cdots \lbar{v}_{m-1}(\lc \lbar{u} \rc\diamond \lbar{v}_{m}')^{-1}\big)\Big)  \Big(\big\lc \lbar{u} \diamond (\lc \lbar{u} \rc\diamond \lbar{v}_{1})\lbar{v}_{2}\cdots \lbar{v}_{m-1}(\lc \lbar{u} \rc\diamond \lbar{v}_{m}')^{-1}\big\rc\diamond \lbar{w}_{1}\Big) \\
&\ \lbar{w}_{2}\cdots \lbar{w}_{n-1}\big(\big\lc \lbar{u}\diamond (\lc \lbar{u} \rc\diamond \lbar{v}_{1})\lbar{v}_{2}\cdots \lbar{v}_{m-1}(\lc \lbar{u} \rc\diamond \lbar{v}_{m}')^{-1}\big\rc\diamond \lbar{w}_{n}'\big)^{-1}\bigg\rc   \hspace{1cm} (\text{Eqs.~(\ref{eq:conc}) and~(\ref{eq:stdia})}).
\end{align*}
On the other hand,
\begin{align*}
&\ u\diamond (v\diamond w) \\
=&\ \lc \lbar{u} \rc\diamond (\lc \lbar{v} \rc\diamond \lc \lbar{w} \rc)\\
=&\ \lc \lbar{u} \rc\diamond \Big\lc \lbar{v}\diamond \big((\lc \lbar{v} \rc\diamond \lbar{w}_{1})\lbar{w}_{2}\cdots \lbar{w}_{n-1}(\lc \lbar{v} \rc\diamond \lbar{w}_{n}')^{-1}\big)\Big\rc \hspace{0.3cm} (\text{Eqs.~(\ref{eq:dia}) and~(\ref{eq:diam})})\\
=&\ \bigg\lc \lbar{u}\diamond \Big((\lc \lbar{u} \rc\diamond \lbar{v}_{1})\lbar{v}_{2}\cdots \lbar{v}_{m}\diamond (\lc \lbar{v} \rc\diamond \lbar{w}_{1})\lbar{w}_{2}\cdots \lbar{w}_{n-1}\big(\lc \lbar{u} \rc\diamond (\lc \lbar{v} \rc\diamond \lbar{w}_{n}')\big)^{-1}\Big)\bigg\rc \\
=&\ \bigg\lc \lbar{u}\diamond \Bigg((\lc \lbar{u} \rc\diamond \lbar{v}_{1})\lbar{v}_{2}\cdots \lbar{v}_{m}\diamond \bigg( \big(\lc \lbar{u} \rc^{-1}\lc \lbar{u} \rc\diamond (\lc \lbar{v} \rc\diamond \lbar{w}_{1})\big)\lbar{w}_{2}\cdots \lbar{w}_{n-1}\big((\lc \lbar{u} \rc\diamond \lc \lbar{v} \rc)\diamond \lbar{w}_{n}'\big)^{-1}\bigg)\Bigg)\bigg\rc\\
&\ \hspace{6cm} (\text{inserting }1=\lc \lbar{u} \rc^{-1}\lc \lbar{u} \rc \text{ and induction on depth})\\
=&\ \bigg\lc \lbar{u}\diamond \Bigg((\lc \lbar{u} \rc\diamond \lbar{v}_{1})\lbar{v}_{2}\cdots \lbar{v}_{m}\diamond \bigg(\Big(\lc \lbar{u} \rc^{-1}\big(\lc \lbar{u} \rc\diamond (\lc \lbar{v} \rc\diamond \lbar{w}_{1})\big)\Big)\lbar{w}_{2}\cdots \lbar{w}_{n-1}\big((\lc \lbar{u} \rc\diamond \lc \lbar{v} \rc)\diamond \lbar{w}_{n}'\big)^{-1}\bigg)\Bigg)\bigg\rc\\
&\ \hspace{10cm} (\text{Eq.~(\ref{eq:stdia})}) \\
=&\ \bigg\lc \lbar{u}\diamond \Bigg((\lc \lbar{u} \rc\diamond \lbar{v}_{1})\lbar{v}_{2}\cdots \lbar{v}_{m}\diamond \bigg(\Big(\lc \lbar{u} \rc^{-1}\big((\lc \lbar{u} \rc\diamond \lc \lbar{v} \rc)\diamond \lbar{w}_{1}\big)\Big)\lbar{w}_{2}\cdots \lbar{w}_{n-1}\big((\lc \lbar{u} \rc\diamond \lc \lbar{v} \rc)\diamond \lbar{w}_{n}'\big)^{-1}\bigg)\Bigg)\bigg\rc \\
&\ \hspace{9cm} (\text{induction on depth})\\
=&\ \bigg\lc \lbar{u}\diamond \bigg((\lc \lbar{u} \rc\diamond \lbar{v}_{1})\lbar{v}_{2}\cdots \lbar{v}_{m-1}(\lbar{v}_{m}\diamond \lc \lbar{u} \rc^{-1})\big((\lc \lbar{u} \rc\diamond \lc \lbar{v} \rc)\diamond \lbar{w}_{1}\big)\lbar{w}_{2}\cdots \lbar{w}_{n-1}\big((\lc \lbar{u} \rc\diamond \lc \lbar{v} \rc)\diamond \lbar{w}_{n}'\big)^{-1}\bigg)\bigg\rc\\
&\ \hspace{10cm}(\text{Eq.~(\ref{eq:stdia})})\\
=&\ \bigg\lc\Big(\lbar{u}\diamond \big((\lc \lbar{u} \rc\diamond \lbar{v}_{1})\lbar{v}_{2}\cdots \lbar{v}_{m-1}(\lc \lbar{u} \rc \diamond \lbar{v}_{m}')^{-1}\big)\Big)  \Big(\big\lc \lbar{u} \diamond (\lc \lbar{u} \rc\diamond \lbar{v}_{1})\lbar{v}_{2}\cdots \lbar{v}_{m-1}(\lc \lbar{u} \rc\diamond \lbar{v}_{m}')^{-1} \big\rc\diamond \lbar{w}_{1}\Big) \\
&\ \lbar{w}_{2}\cdots \lbar{w}_{n-1}\big(\big\lc \lbar{u}\diamond (\lc \lbar{u} \rc\diamond \lbar{v}_1)\lbar{v}_{2}\cdots \lbar{v}_{m-1}(\lc \lbar{u} \rc\diamond \lbar{v}_{m}')^{-1}\big\rc\diamond \lbar{w}_{n}'\big)^{-1}\bigg\rc\\
&\ \hspace{5cm}\big(\lbar{v}_{m} = ( \lbar{v}_{m}')^{-1}\text{and Eqs.~(\ref{eq:dia}), (\ref{eq:diam}), (\ref{eq:uvconc}) and (\ref{eq:stdia})}\big).
\end{align*}
So Eq.~(\mref{eq:as}) holds.

\noindent{\bf Case 1.2.} Suppose $\lbar{v}_{m}=(\lbar{v}_{m}')^{-1}\in \lc \rbgw\rc^{-1}$ for some $\lbar{v}_{m}'\in \lc\rbgw\rc$ and $\lbar{w}_{n}\notin \lc \rbgw\rc^{-1}$. Then
\begin{align*}
&\ (u\diamond v) \diamond w \\
=&\ (\lc \lbar{u} \rc\diamond \lc \lbar{v} \rc)\diamond \lc \lbar{w} \rc\\
=&\ \Big\lc \lbar{u}\diamond \big((\lc \lbar{u} \rc\diamond \lbar{v}_{1})\lbar{v}_{2}\cdots \lbar{v}_{m-1}(\lc \lbar{u} \rc\diamond \lbar{v}_{m}')^{-1}\big)\Big\rc\diamond \lc \lbar{w} \rc \hspace{1cm} (\text{Eqs.~(\ref{eq:dia}) and~(\ref{eq:diam})})\\
=&\ \bigg\lc \Big(\lbar{u}\diamond \big((\lc \lbar{u} \rc\diamond \lbar{v}_{1})\lbar{v}_{2}\cdots \lbar{v}_{m-1}(\lc \lbar{u} \rc\diamond \lbar{v}_{m}')^{-1}\big)\Big)\diamond \Big(\Big\lc \lbar{u}\diamond \big((\lc \lbar{u} \rc\diamond \lbar{v}_{1})\lbar{v}_{2}\cdots \lbar{v}_{m-1}(\lc \lbar{u} \rc\diamond \lbar{v}_{m}')^{-1}\big)\Big\rc\diamond \lbar{w}\Big) \\
&\ \Big\lc \lbar{u}\diamond \big((\lc \lbar{u} \rc\diamond \lbar{v}_{1})\lbar{v}_{2}\cdots \lbar{v}_{m-1}(\lc \lbar{u} \rc\diamond \lbar{v}_{m}')^{-1}\big)\Big\rc^{-1}\bigg\rc \hspace{1cm} (\text{Eqs.~(\ref{eq:dia}) and~(\ref{eq:diam})}) \\
=&\ \bigg\lc \Big(\lbar{u}\diamond \big((\lc \lbar{u} \rc\diamond \lbar{v}_{1})\lbar{v}_{2}\cdots \lbar{v}_{m-1}(\lc \lbar{u} \rc\diamond \lbar{v}_{m}')^{-1}\big)\Big)  \Big(\Big\lc \lbar{u}\diamond \big((\lc \lbar{u} \rc\diamond \lbar{v}_{1})\lbar{v}_{2}\cdots \lbar{v}_{m-1}(\lc \lbar{u} \rc\diamond \lbar{v}_{m}')^{-1}\big)\Big\rc\diamond \lbar{w}\Big) \\
&\ \Big\lc \lbar{u}\diamond \big((\lc \lbar{u} \rc\diamond \lbar{v}_{1})\lbar{v}_{2}\cdots \lbar{v}_{m-1}(\lc \lbar{u} \rc\diamond \lbar{v}_{m}')^{-1}\big)\Big\rc^{-1}\bigg\rc   \hspace{3cm} (\text{Eqs.~(\ref{eq:conc}) and~(\ref{eq:stdia})})
\end{align*}
and
\begin{align*}
&\ u\diamond (v \diamond w ) \\
= &\ \lc u \rc\diamond (\lc v \rc\diamond \lc w \rc) \\
=&\ \lc u \rc\diamond \Big\lc \lbar{v} \diamond \Big( (\lc \lbar{v}\rc \diamond \lbar{w})\lc \lbar{v}\rc^{-1}  \Big)   \Big\rc \hspace{1cm} (\text{Eqs.~(\ref{eq:dia}) and~(\ref{eq:diam})})\\
=&\ \lc \lbar{u} \rc\diamond \Big\lc \lbar{v}_{1}\Big(\lbar{v}_{2}\cdots \lbar{v}_{m}\diamond (\lc \lbar{v} \rc\diamond \lbar{w})\Big)\lc \lbar{v} \rc^{-1}\Big\rc \hspace{2cm} (\text{Eq.~(\ref{eq:stdia})})\\
=&\ \bigg\lc \lbar{u}\diamond \bigg((\lc \lbar{u} \rc\diamond \lbar{v}_{1})\Big(\lbar{v}_{2}\cdots \lbar{v}_{m}\diamond (\lc \lbar{v} \rc\diamond \lbar{w})\Big)(\lc \lbar{u} \rc\diamond \lc \lbar{v} \rc)^{-1}\bigg)\bigg\rc \hspace{1cm} (\text{Eqs.~(\ref{eq:dia}) and (\ref{eq:diam})})\\
=&\ \bigg\lc \lbar{u}\diamond \bigg((\lc \lbar{u} \rc\diamond \lbar{v}_{1})\Big(\lbar{v}_{2}\cdots \lbar{v}_{m}\diamond \big((\lc \lbar{u}\rc^{-1} \lc \lbar{u}\rc)\diamond (\lc \lbar{v} \rc\diamond \lbar{w})\big) \Big)(\lc \lbar{u} \rc\diamond \lc \lbar{v} \rc)^{-1}\bigg)\bigg\rc  \quad\quad (1= \lc \lbar{u}\rc^{-1} \lc \lbar{u}\rc )\\
=&\ \bigg\lc \lbar{u}\diamond \Bigg((\lc \lbar{u} \rc\diamond \lbar{v}_{1})\bigg((\lbar{v}_{2}\cdots \lbar{v}_{m}\diamond \lc \lbar{u} \rc^{-1})\Big(\lc \lbar{u} \rc\diamond (\lc \lbar{v} \rc\diamond \lbar{w})\Big)\bigg)(\lc \lbar{u} \rc\diamond \lc \lbar{v} \rc)^{-1}\Bigg)\bigg\rc \quad (\text{Eq.~(\ref{eq:stdia})})\\
=&\ \bigg\lc \lbar{u}\diamond \Bigg((\lc \lbar{u} \rc\diamond \lbar{v}_{1})\bigg((\lbar{v}_{2}\cdots \lbar{v}_{m}\diamond \lc \lbar{u} \rc^{-1})\Big((\lc \lbar{u} \rc\diamond \lc \lbar{v} \rc)\diamond \lbar{w}\Big)\bigg)(\lc \lbar{u} \rc\diamond \lc \lbar{v} \rc)^{-1}\Bigg)\bigg\rc \hspace{0.5cm} (\text{induction on depth})\\
=&\ \bigg\lc \lbar{u}\diamond \Big((\lc \lbar{u} \rc\diamond \lbar{v}_{1})\lbar{v}_{2}\cdots \lbar{v}_{m-1}( \lc \lbar{u} \rc \diamond \lbar{v}'_{m})^{-1}\Big)\bigg(\Big\lc \lbar{u} \diamond \Big((\lc \lbar{u} \rc\diamond \lbar{v}_{1})\lbar{v}_{2}\cdots \lbar{v}_{m-1}(\lc \lbar{u} \rc\diamond \lbar{v}_{m}')^{-1}\Big)\Big\rc\diamond \lbar{w}\bigg)\\
&\ \Big\lc \lbar{u} \diamond \Big((\lc \lbar{u} \rc\diamond \lbar{v}_{1})\lbar{v}_{2}\cdots \lbar{v}_{m-1}( \lc \lbar{u} \rc\diamond \lbar{v}_{m}')^{-1}\Big)\Big\rc^{-1}\bigg\rc
\hspace{0.8cm} \quad (\lbar{v}_{m} = ( \lbar{v}_{m}')^{-1} \text{and Eqs.~(\ref{eq:dia}), (\ref{eq:diam}), (\ref{eq:stdia})}),
\end{align*}
showing that Eq.~(\mref{eq:as}) is valid in this case.

\noindent{\bf Case 1.3.}  Suppose $\lbar{v}_{m}\notin\lc\rbgw\rc^{-1}$ and $\lbar{w}_{n}=(\lbar{w}_{n}')^{-1} \in \lc\rbgw\rc^{-1}$ for some $\lbar{w}_{n}'\in \lc\rbgw\rc$. Then we check
\begin{align*}
&\ (u\diamond v)\diamond w \\
= &\ (\lc u \rc\diamond \lc v \rc)\diamond \lc w \rc\\
=&\ \Big\lc \lbar{u}\diamond \Big((\lc \lbar{u} \rc\diamond \lbar{v})\lc \lbar{u} \rc^{-1}\Big)\Big\rc\diamond \lc \lbar{w} \rc\hspace{4cm} (\text{Eqs.~(\ref{eq:dia}) and~(\ref{eq:diam})})\\
=&\ \bigg\lc \bigg(\lbar{u}\diamond \Big((\lc \lbar{u} \rc\diamond \lbar{v})\lc \lbar{u} \rc^{-1}\Big)\bigg)\diamond \bigg(\Big\lc \lbar{u}\diamond \Big((\lc \lbar{u} \rc\diamond \lbar{v})\lc \lbar{u} \rc^{-1}\Big)\Big\rc \diamond \lbar{w}_{1}\bigg) \\
&\ \lbar{w}_{2}\cdots \lbar{w}_{m-1} \bigg(\Big\lc \lbar{u}\diamond \Big((\lc \lbar{u} \rc\diamond \lbar{v})\lc \lbar{u} \rc^{-1}\Big)\Big\rc \diamond \lbar{w}_{n}'\bigg)^{-1}\bigg\rc\hspace{2cm} (\text{Eqs.~(\ref{eq:dia}) and~(\ref{eq:diam})})\\
=&\ \bigg\lc \bigg(\lbar{u}\diamond \Big((\lc \lbar{u} \rc\diamond \lbar{v})\lc \lbar{u} \rc^{-1}\Big)\bigg)  \bigg(\Big\lc \lbar{u}\diamond \Big((\lc \lbar{u} \rc\diamond \lbar{v})\lc \lbar{u} \rc^{-1}\Big)\Big\rc \diamond \lbar{w}_{1}\bigg) \\
&\ \lbar{w}_{2}\cdots \lbar{w}_{m-1} \bigg(\Big\lc \lbar{u}\diamond \Big((\lc \lbar{u} \rc\diamond \lbar{v})\lc \lbar{u} \rc^{-1}\Big)\Big\rc \diamond \lbar{w}_{n}'\bigg)^{-1}\bigg\rc  \hspace{1cm} (\text{Eqs.~(\ref{eq:conc}) and~(\ref{eq:stdia})}),
\end{align*}
and
\begin{align*}
&\ u\diamond(v\diamond w) \\
= &\ \lc u \rc\diamond (\lc v \rc\diamond \lc w \rc)\\
=&\ \lc \lbar{u} \rc \diamond \Big\lc \lbar{v}\diamond \Big((\lc \lbar{v} \rc\diamond \lbar{w}_{1})\lbar{w}_{2}\cdots \lbar{w}_{n-1}(\lc \lbar{v} \rc\diamond \lbar{w}_{n}')^{-1}\Big)\Big\rc \hspace{2cm} (\text{Eqs.~(\ref{eq:dia}) and (\ref{eq:diam})})\\
=&\ \bigg\lc \lbar{u}\diamond \bigg(\lc \lbar{u} \rc\diamond \Big(\lbar{v}\diamond (\lc \lbar{v} \rc\diamond \lbar{w}_{1})\lbar{w}_{2}\cdots \lbar{w}_{n-1}\Big)\Big( \lc \lbar{u} \rc\diamond (\lc \lbar{v} \rc\diamond \lbar{w}_{n}')\Big)^{-1}\bigg)\bigg\rc\quad (\text{Eqs.~(\ref{eq:dia}), (\ref{eq:diam}) and~(\ref{eq:stdia})})\\
=&\ \bigg\lc \lbar{u}\diamond \bigg(\lc \lbar{u} \rc\diamond \Big((\lbar{v}\lc \lbar{u} \rc^{-1}\lc \lbar{u} \rc)\diamond (\lc \lbar{v} \rc\diamond \lbar{w}_{1})\lbar{w}_{2}\cdots \lbar{w}_{n-1}\Big)\Big( \lc \lbar{u} \rc\diamond (\lc \lbar{v} \rc\diamond \lbar{w}_{n}')\Big)^{-1}\bigg)\bigg\rc \hspace{0.1cm} (\text{inserting }1=\lc \lbar{u} \rc^{-1}\lc \lbar{u} \rc)\\
=&\ \bigg\lc \lbar{u}\diamond \Bigg(\lc \lbar{u} \rc\diamond \bigg((\lbar{v}\lc \lbar{u} \rc^{-1})\Big(\lc \lbar{u} \rc\diamond (\lc \lbar{v} \rc\diamond \lbar{w}_{1})\Big)\lbar{w}_{2}\cdots \lbar{w}_{n-1}\bigg)\Big( \lc \lbar{u} \rc\diamond (\lc \lbar{v} \rc\diamond \lbar{w}_{n}')\Big)^{-1}\Bigg)\bigg\rc \quad (\text{Eq.~(\ref{eq:stdia})})\\
=&\ \bigg\lc \lbar{u}\diamond \Bigg(\lc \lbar{u} \rc\diamond \bigg((\lbar{v}\lc \lbar{u} \rc^{-1})\Big((\lc \lbar{u} \rc\diamond \lc \lbar{v} \rc)\diamond \lbar{w}_{1}\Big)\lbar{w}_{2}\cdots \lbar{w}_{n-1}\bigg)\Big( (\lc \lbar{u} \rc\diamond \lc \lbar{v} \rc)\diamond \lbar{w}_{n}'\Big)^{-1}\Bigg)\bigg\rc \hspace{0.5cm} (\text{induction on depth})\\
=&\ \bigg\lc \bigg(\lbar{u}\diamond \Big((\lc \lbar{u} \rc\diamond \lbar{v})\lc \lbar{u} \rc^{-1}\Big)\bigg)  \bigg(\Big\lc \lbar{u}\diamond \Big((\lc \lbar{u} \rc\diamond \lbar{v})\lc \lbar{u} \rc^{-1}\Big)\Big\rc \diamond \lbar{w}_{1}\bigg) \\
&\ \lbar{w}_{2}\cdots \lbar{w}_{m-1} \bigg(\Big\lc \lbar{u}\diamond \Big((\lc \lbar{u} \rc\diamond \lbar{v})\lc \lbar{u} \rc^{-1}\Big)\Big\rc \diamond \lbar{w}_{n}'\bigg)^{-1}\bigg\rc \hspace{3cm} (\text{Eqs.~(\ref{eq:dia}), (\ref{eq:diam}) and~(\ref{eq:stdia})}).
\end{align*}
Hence Eq.~(\mref{eq:as}) holds.

\noindent{\bf Case 1.4.}
Suppose $\lbar{v}_{m}, \lbar{w}_{n}\notin \lc \rbgw\rc^{-1}$. Then we have 
\begin{align*}
&\ (u\diamond v)\diamond w \\
=&\ (\lc u \rc\diamond \lc v \rc)\diamond \lc w \rc\\
=&\ \Big\lc \lbar{u}\diamond \big( (\lc \lbar{u} \rc\diamond \lbar{v})\lc \lbar{u} \rc^{-1} \big)\Big \rc \diamond \lc \lbar{w} \rc\hspace{4cm} (\text{Eqs.~(\ref{eq:dia}) and (\ref{eq:diam})})\\
=&\ \bigg\lc \Big( \lbar{u}\diamond ( (\lc \lbar{u} \rc\diamond \lbar{v})\lc \lbar{u} \rc^{-1})\Big)\diamond \bigg(\Big(\Big\lc \lbar{u}\diamond (\lc \lbar{u} \rc\diamond \lbar{v})\lc \lbar{u} \rc^{-1}\Big\rc\diamond \lbar{w}\Big)\Big\lc \lbar{u}\diamond (\lc \lbar{u} \rc\diamond \lbar{v})\lc \lbar{u} \rc^{-1}\Big\rc^{-1}\bigg)\bigg\rc\\
&\ \hspace{9cm} (\text{Eqs.~(\ref{eq:dia}) and (\ref{eq:diam})})\\
=&\ \bigg\lc \Big( \lbar{u}\diamond ( (\lc \lbar{u} \rc\diamond \lbar{v})\lc \lbar{u} \rc^{-1})\Big) \bigg(\Big(\Big\lc \lbar{u}\diamond (\lc \lbar{u} \rc\diamond \lbar{v})\lc \lbar{u} \rc^{-1}\Big\rc\diamond \lbar{w}\Big)\Big\lc \lbar{u}\diamond (\lc \lbar{u} \rc\diamond \lbar{v})\lc \lbar{u} \rc^{-1}\Big\rc^{-1}\bigg)\bigg\rc\\
&\ \hspace{9cm} (\text{Eqs.~(\ref{eq:conc}) and~(\ref{eq:stdia})}),
\end{align*}
and also
\begin{align*}
&\ u\diamond(v\diamond w) \\
= &\ \lc u \rc\diamond (\lc v \rc\diamond \lc w \rc)\\
=&\ \lc \lbar{u} \rc\diamond \Big\lc \lbar{v}\diamond (\lc \lbar{v} \rc\diamond \lbar{w})\lc \lbar{v} \rc^{-1}\Big\rc\hspace{5.3cm} (\text{Eqs.~(\ref{eq:dia}) and (\ref{eq:diam})})\\
=&\ \bigg\lc \lbar{u}\diamond \bigg(\lc \lbar{u} \rc\diamond \Big(\lbar{v}\diamond (\lc \lbar{v} \rc\diamond \lbar{w})\Big)(\lc\lbar{u}\rc \diamond \lc \lbar{v}\rc)^{-1}\bigg) \bigg\rc \hspace{3cm} (\text{Eqs.(\ref{eq:dia}) and (\ref{eq:diam})})\\
=&\ \bigg\lc \lbar{u}\diamond \bigg(\lc \lbar{u} \rc\diamond \Big((\lbar{v}\lc \lbar{u} \rc^{-1}\lc \lbar{u} \rc)\diamond (\lc \lbar{v} \rc\diamond \lbar{w})\Big) (\lc\lbar{u}\rc \diamond \lc \lbar{v}\rc)^{-1}\bigg) \bigg\rc \quad (\text{inserting }1=\lc \lbar{u} \rc^{-1}\lc \lbar{u} \rc)\\
=&\ \bigg\lc \lbar{u}\diamond \Bigg(\lc \lbar{u} \rc\diamond \bigg(\lbar{v}\lc \lbar{u} \rc^{-1}\Big(\lc \lbar{u} \rc\diamond (\lc \lbar{v} \rc\diamond \lbar{w})\Big)\bigg) (\lc\lbar{u}\rc \diamond \lc \lbar{v}\rc)^{-1} \Bigg)\bigg\rc \hspace{2cm} (\text{Eq.~(\ref{eq:stdia})})\\
=&\ \bigg\lc \lbar{u}\diamond \Bigg(\lc \lbar{u} \rc\diamond \bigg(\lbar{v}\lc \lbar{u} \rc^{-1}\Big((\lc \lbar{u} \rc\diamond \lc \lbar{v} \rc)\diamond \lbar{w}\Big)\bigg)  (\lc\lbar{u}\rc \diamond \lc \lbar{v}\rc)^{-1}  \Bigg) \bigg\rc   \hspace{1cm} (\text{induction on depth})\\
=&\ \bigg\lc \lbar{u}\diamond \Bigg(\lc \lbar{u} \rc\diamond \bigg(\lbar{v}\lc \lbar{u} \rc^{-1}\Big(\Big\lc \lbar{u}\diamond (\lc \lbar{u}\rc\diamond \lbar{v})\lc \lbar{u} \rc^{-1}\Big\rc\diamond \lbar{w}\Big)\bigg) (\lc\lbar{u}\rc \diamond \lc \lbar{v}\rc)^{-1} \Bigg)\bigg\rc  \quad (\text{Eqs.~(\ref{eq:dia}) and (\ref{eq:diam})})\\
=&\ \bigg\lc \lbar{u}\diamond \Bigg((\lc \lbar{u} \rc\diamond \lbar{v}\lc \lbar{u} \rc^{-1})\Big(\Big\lc \lbar{u}\diamond (\lc \lbar{u}\rc\diamond \lbar{v})\lc \lbar{u} \rc^{-1}\Big\rc\diamond \lbar{w}\Big)  (\lc\lbar{u}\rc \diamond \lc \lbar{v}\rc)^{-1} \Bigg)\bigg\rc
\hspace{2cm} (\text{Eq.~(\ref{eq:stdia})})\\
=&\ \bigg\lc \Big(\lbar{u}\diamond (\lc \lbar{u} \rc\diamond \lbar{v}\lc \lbar{u} \rc^{-1})\Big)\Big(\Big\lc \lbar{u}\diamond (\lc \lbar{u} \rc\diamond \lbar{v})\lc \lbar{u} \rc^{-1}\Big\rc\diamond \lbar{w}\Big)(\lc \lbar{u} \rc\diamond \lc \lbar{v} \rc)^{-1}\bigg\rc\hspace{2cm} (\text{Eq.~(\ref{eq:stdia})})\\
=&\ \bigg\lc \Big(\lbar{u}\diamond ( (\lc \lbar{u} \rc\diamond \lbar{v}) \lc \lbar{u} \rc^{-1})\Big) \bigg(\Big(\Big\lc \lbar{u}\diamond (\lc \lbar{u} \rc\diamond \lbar{v})\lc \lbar{u} \rc^{-1}\Big\rc\diamond \lbar{w}\Big)\Big\lc \lbar{u}\diamond (\lc \lbar{u} \rc\diamond \lbar{v})\lc \lbar{u} \rc^{-1}\Big\rc^{-1}\bigg)\bigg\rc\\
& \hspace{8cm} (\text{Eqs.~(\ref{eq:dia}), (\ref{eq:diam}) and~(\ref{eq:stdia})}).
\end{align*}
Therefore Eq.~(\mref{eq:as}) is valid.

\noindent{\bf Case 2.} Suppose $u\notin\lc \rbgw\rc$.
Then $u\in X$ or $u\in \lc \rbgw\rc^{-1}$.
If $u\in X$, then
\begin{equation}
(u\diamond v)\diamond w=\ (uv)\diamond w=\ u(v\diamond w)=\ u\diamond (v\diamond w).
\mlabel{eq:asso1}
\end{equation}
Suppose $u\notin X$, i.e., $u\in \lc \rbgw\rc^{-1}$.
If $v\notin\lc \rbgw\rc^{-1}$, then Eq.~(\ref{eq:asso1}) holds again.
Assume $v\in \lc \rbgw\rc^{-1}$. If $w\notin \lc \rbgw\rc^{-1}$, then
\begin{equation}
u\diamond (v\diamond w)=u\diamond (vw)=(u\diamond v)w=(u\diamond v)\diamond w.
\mlabel{eq:asso2}
\end{equation}
If $w\in \lc \rbgw\rc^{-1}$, let $u=(u')^{-1}, v=(v')^{-1}$ and $w=(w')^{-1}$ for some $u', v', w'\in\lc \rbgw\rc$. On the one hand,
\begin{align*}
(u\diamond v)=\Big((u')^{-1}\diamond (v')^{-1}\Big)\diamond (w')^{-1}=(v'\diamond u')^{-1}\diamond (w')^{-1}=\Big(w'\diamond (v'\diamond u')\Big)^{-1}.
\end{align*}
On the other hand,
\begin{align*}
u\diamond (v\diamond w)=(u')^{-1}\diamond \Big((v')^{-1}\diamond (w')^{-1}\Big)=(u')^{-1}\diamond (w'\diamond v')^{-1}=\Big((w'\diamond v')\diamond u'\Big)^{-1}.
\end{align*}
Then the proof can be reduced to Case 1.

\noindent{\bf Case 3.}  Suppose $w\notin\lc \rbgw\rc$. The proof is similar to Case 2.

\noindent{\bf Case 4.} Suppose $v\notin\lc \rbgw\rc$. Then $v\in X$ or $v\in\lc \rbgw\rc^{-1}$.
This case is similar to Case 2 by symmetry.

This completes the initial step of the induction on  $\bre(u)+\bre(v)+\bre(w)$.

\smallskip

\noindent
{\bf II. The inductive step on $\bre(u)+\bre(v)+\bre(w)$.}
For a given $k\geq 3$, suppose that Eq.~\meqref{eq:as} has been verified for all $u, v, w\in \rbgw$ with $\bre(u)+\bre(v)+\bre(w)\leq k$ and consider the case when $\bre(u)+\bre(v)+\bre(w)=k+1$.
Then $k+1\geq 4$ and so at least one of $\bre(u)$, $\bre(v)$ or $\bre(w)$ is greater than one. We accordingly have the following three cases to verify.

\noindent{\bf Case 1.} Suppose $\bre(u)\geq 1$. Then we write
$u=u_{1}\cdots u_{m}$ in the standard factorization with $m\geq 2$. Then by Eq.~\meqref{eq:stdia} and the induction hypothesis we have
\begin{align*}
(u\diamond v)\diamond w=&\ \big((u_{1}\cdots u_{m}) \diamond v\big) \diamond w\\
=&\ \big((u_{1}\cdots u_{m-1})(u_{m}\diamond v)\big)\diamond w   \\
=&\ (u_{1}\cdots u_{m-1})\big((u_{m}\diamond v)\diamond w\big) \\
=&\ (u_{1}\cdots u_{m-1})\big(u_{m}\diamond (v\diamond w)\big) \\
=&\ (u_{1}\cdots u_{m-1}u_{m})\diamond (v \diamond w) \\
=&\ u\diamond (v\diamond w).
\end{align*}

\noindent{\bf Case 2.}  Suppose $\bre(w)\geq 2$. Then the verification is similar to the previous case.

\noindent{\bf Case 3.} Suppose $\bre(v)\geq 2$. Write $v=v_{1}\cdots v_{n}$ in the standard factorization with $n\geq 2$. Then repeatedly applying Eq.~\meqref{eq:stdia}, we obtain
\begin{align*}
(u\diamond v)\diamond w=&\ \big(u\diamond (v_{1}\cdots v_{n})\big)\diamond w\\
=&\ \big( (u\diamond v_{1}) v_{2}\cdots v_{n}\big)\diamond w\\
=&\ (u\diamond v_{1})(v_{2}\cdots v_{n-1})(v_{n}\diamond w)\\
=&\ u\diamond \big( (v_{1}\cdots v_{n-1})(v_{n}\diamond w) \big) \\
=&\ u\diamond \big((v_{1}\cdots v_{n})\diamond w\big)\\
=&\ u\diamond (v\diamond w).
\end{align*}

Now we have completed the inductive proof of the associativity in Eq.~\meqref{eq:as}.

\smallskip 

Finally, we verify Eq.~(\mref{RBgroup}).
Let $u=\lc \lbar{u} \rc$ and $v=\lc \lbar{v} \rc$ be in $\lc \rbgw \rc$. Write $\lbar{v}=\lbar{v}_{1}\cdots \lbar{v}_{k}$ in the standard form. If $\lbar{v}_{k}=(\lbar{v}_{k}')^{-1}\in \lc \rbgw \rc^{-1}$, we have
\begin{align*}
\AD_{u}\lbar{v}=&\ (u\diamond \lbar{v}_{1})\lbar{v}_{2}\cdots \lbar{v}_{k-1}(u \diamond \lbar{v}_{k}')^{-1} \hspace{3cm} (\text{Eq.~(\mref{eq:diam})})\\
=&\ (u\diamond \lbar{v}_{1})\lbar{v}_{2}\cdots \lbar{v}_{k-1}(\lbar{v}_{k} \diamond u^{-1})\\
=&\ \Big((u\diamond \lbar{v}_{1})\lbar{v}_{2}\cdots \lbar{v}_{k-1}\lbar{v}_{k}\Big) \diamond u^{-1} \hspace{3cm} (\text{Eq.~(\mref{eq:stdia})})\\
=&\ u \diamond v \diamond u^{-1} \hspace{5.7cm} (\text{Eq.~(\mref{eq:stdia})}).
\end{align*}
If $\lbar{v}_{k}\notin \lc \rbgw \rc^{-1}$, we obtain
\begin{align*}
\AD_{u}\lbar{v}=&\ (u\diamond \lbar{v})u^{-1}\hspace{5.3cm} (\text{Eq.~(\mref{eq:diam})})\\
=&\ \Big( ( u \diamond \lbar{v}_{1})\lbar{v}_{2}\cdots \lbar{v_{k}} \Big)u^{-1}\hspace{3.6cm} (\text{Eq.~(\mref{eq:stdia})})\\
=&\ (u\diamond \lbar{v}_{1})\lbar{v}_{2}\cdots \lbar{v}_{k-1}(\lbar{v}_{k}\diamond u^{-1})\hspace{2.6cm} (\text{Eq.~(\mref{eq:conc})})\\
=&\ \Big( (u \diamond \lbar{v}_{1})\lbar{v}_{2}\cdots \lbar{v}_{k} \Big)\diamond u^{-1} \hspace{3.2cm} (\text{Eq.~(\mref{eq:stdia})})\\
=&\ u \diamond v\diamond u^{-1}\hspace{5.2cm} (\text{Eq.~(\mref{eq:stdia})}).
\end{align*}
Hence
\begin{align*}
u \diamond v = \lc \lbar{u} \diamond \AD_{u}\lbar{v}\rc= \lc \lbar{u} \diamond u \diamond v \diamond u^{-1} \rc,
\end{align*}
and so Eq.~(\mref{RBgroup}) holds. This completes the proof.
\end{proof}

\begin{thm}
Let $X$ be a set. The triple $(\rbgw, \rbgo, \diamond)$, together with the inclusion $i_{X}:X\rightarrow \rbgw$, is a free Rota-Baxter group of weight 1 on $X$.
\mlabel{thm:rtgp2}
\end{thm}

\begin{proof}
Let $(G,P)$ be a Rota-Baxter group and $f:X\to G$ a set map. It suffices to show that there is a unique Rota-Baxter group homomorphism $$\lbar{f}: \rbgw\to G, w\mapsto \lbar{f}(w)$$ such that $f=\lbar{f} i_{X}$. 

({\bf Existence}) We use induction on $\dep(w)\geq 0$ to define $\lbar{f}(w)$.

For the initial step of $\dep(w) = 0$, if $\bre(w)\leq 1$, then $w=1$ or $w\in X\sqcup X^{-1}$. Then define
\begin{align}
\lbar{f}(w):=
\begin{cases}
1_G, & \text{ if } w = 1,\\
f(w),& \text{ if } w \in X,\\
f(x)^{-1}, &\text{ if } w=x^{-1} \in X^{-1} \text{ for some } x\in X.
\end{cases}
\mlabel{eq:bfdep0}
\end{align}
If $\bre(w) \geq 2$, write $w=w_{1}\cdots w_{n}$ in the standard form with $n\geq 2$ and define
\begin{equation}
\lbar{f}(w) :=\lbar{f}(w_{1}\cdots w_{n}) :=\lbar{f}(w_{1})\cdots \lbar{f}(w_{n}),
\mlabel{eq:rbdef3'}
\end{equation}
where each $\lbar{f}(w_i)$ is defined in Eq.~(\mref{eq:bfdep0}).

For a given $k\geq 0$, assume that $\lbar{f}(w)$ have been defined for $w\in \rbgw$ with $\dep(w)\leq k$ and consider $w\in \rbgw$ with $\dep(w)=k+1$.

First consider $w$ with $\bre(w)= 1$. From $\dep(w)=k+1\geq 1$, we have $w\in \lc \rbgw\rc$ or $w\in \lc \rbgw\rc^{-1}$. In the first case, write $w=\lc \lbar{w}\rc$ for some $\lbar{w}\in \rbgw$ with $\dep(\lbar{w})=k$. So the induction hypothesis allows us to define
\begin{equation}
\lbar{f}(w):=\lbar{f}(\lc \lbar{w}\rc):= P\big(\lbar{f}(\lbar{w})\big).
\mlabel{eq:rbdef1}
\end{equation}
In the second case, write $w =\lc \lbar{w}\rc^{-1}$ for some $\lbar{w}\in \rbgw$ with $\lbar{w}=k$. Then the inductive hypothesis allows us to define
\begin{equation}
\lbar{f}(w) :=\lbar{f}(\lc \lbar{w}\rc^{-1}):=P\big(\lbar{f}(\lbar{w})\big)^{-1}.
\mlabel{eq:rbdef2}
\end{equation}

Next consider $w$ with $\bre(w)\geq 2$. Write $w=w_{1}\cdots w_{n}$ in the standard form with $n\geq 2$ and define
\begin{equation}
\lbar{f}(w) :=\lbar{f}(w_{1}\cdots w_{n}) :=\lbar{f}(w_{1})\cdots \lbar{f}(w_{n}),
\mlabel{eq:rbdef3}
\end{equation}
where each $\lbar{f}(w_i)$ is defined in Eqs.~(\mref{eq:bfdep0}), (\mref{eq:rbdef1}) or~(\mref{eq:rbdef2}).

By Eq.~(\ref{eq:rbdef1}), we have $\lbar{f}  \circ \lc\,\rc = P\circ \lbar{f}$. So we are left to prove that $\lbar{f}$ is a group homomorphism:
\begin{equation}
\lbar{f}(u\diamond v)=\lbar{f}(u)\lbar{f}(v)\tforall~ u,v\in \rbgw,
\mlabel{eq:rbhomo}
\end{equation}
which will be achieved by an induction on $(\dep(u), \dep(v))\geq (0,0)$ lexicographically.

If $(\dep(u), \dep(v))=(0, 0)$, then $u, v\in \calg_{0}$ and Eq.~\eqref{eq:rbhomo} follows from Eq.~\eqref{eq:rbdef3'}.
Next fix $m,n\geq 0$ with $(m,n)>(0,0)$ lexicographically.
Assume that Eq.~\eqref{eq:rbhomo} is valid for $u,v\in \rbgw$ with $(\dep(u), \dep(v))< (m, n)$
lexicographically, and consider $u,v\in \rbgw$ with $(\dep(u), \dep(v)) = (m, n)$.
According to the four cases in the definition of $\diamond$, we have the following four cases to verify.

\noindent{\bf Case 1.} Suppose that $u=\lc\lbar{u}\rc$ and $v=\lc\lbar{v}\rc$ for some $\lbar{u}, \lbar{v}\in \rbgw$.
Write $\lbar{v}=\lbar{v}_{1}\cdots\lbar{v}_{k}$ in the standard form. According to $\lbar{v}_k$ in $\lc \rbgw\rc^{-1}$ or not, we consider two subcases.

\noindent{\bf Subcase 1.1.} $\lbar{v}_k = (\lbar{v}'_k)^{-1} \in \lc \rbgw\rc^{-1}$ with $\lbar{v}'_k\in \lc \rbgw\rc$. Then
\begin{align*}
&\ \lbar{f}(u\diamond v)\\
=&\ \lbar{f}(\lc \lbar{u} \rc \diamond \lc \lbar{v}\rc)\\
=&\ \lbar{f} \bigg( \bigg\lc \lbar{u} \diamond \Big((u \diamond \lbar{v}_1) \lbar{v}_2 \cdots \lbar{v}_{k-1} (u\diamond \lbar{v}'_k)^{-1}\Big) \bigg\rc \bigg) \hspace{4cm} (\text{Eqs.~(\ref{eq:dia}) and~(\ref{eq:diam})})\\
=&\ P\bigg( \lbar{f} \bigg(  \lbar{u} \diamond \Big((u \diamond \lbar{v}_1) \lbar{v}_2 \cdots \lbar{v}_{k-1} (u\diamond \lbar{v}'_k)^{-1}\Big) \bigg) \bigg)\hspace{4cm} (\text{Eq.~(\ref{eq:rbdef1})})\\
=&\ P\bigg( \lbar{f}(\lbar{u}) \, \lbar{f}\Big((u \diamond \lbar{v}_1) \lbar{v}_2 \cdots \lbar{v}_{k-1} (u\diamond \lbar{v}'_k)^{-1}\Big)\bigg) \hspace{4cm} (\text{induction on depth})\\
=&\ P\bigg( \lbar{f}(\lbar{u}) \, \lbar{f}(u \diamond \lbar{v}_1)\, \lbar{f}(\lbar{v}_2 \cdots \lbar{v}_{k-1})\, \lbar{f}( (u\diamond \lbar{v}'_k)^{-1}) \bigg) \hspace{4cm} (\text{Eq.~(\ref{eq:rbdef3})})\\
=&\ P\bigg( \lbar{f}(\lbar{u}) \, \lbar{f}(u) \, \lbar{f}(\lbar{v}_{1}) \, \lbar{f}(\lbar{v}_2 \cdots \lbar{v}_{k-1})
\lbar{f}( u\diamond \lbar{v}'_k)^{-1} \bigg) \hspace{1cm} (\text{induction on depth and Eq.~(\ref{eq:rbdef2})} )\\
=&\ P\bigg( \lbar{f}(\lbar{u}) \, \lbar{f}(u) \, \lbar{f}(\lbar{v}_{1}) \, \lbar{f}(\lbar{v}_2 \cdots \lbar{v}_{k-1})
\Big(\lbar{f}( u)  \lbar{f}(\lbar{v}'_k)\Big)^{-1} \bigg) \hspace{2cm} (\text{induction on depth} )\\
=&\ P\bigg( \lbar{f}(\lbar{u}) \, P(\lbar{f}(\lbar{u})) \, \lbar{f}(\lbar{v}_{1}) \, \lbar{f}(\lbar{v}_2 \cdots \lbar{v}_{k-1})
\Big( P(\lbar{f}(\lbar{u})) \, \lbar{f}(\lbar{v}'_k)\Big)^{-1} \bigg) \hspace{1cm} (u =\lc \lbar{u}\rc \text{ and Eq.~(\ref{eq:rbdef1})} )\\
=&\ P\bigg( \lbar{f}(\lbar{u}) \, P(\lbar{f}(\lbar{u})) \, \lbar{f}(\lbar{v}_{1}) \, \lbar{f}(\lbar{v}_2 \cdots \lbar{v}_{k-1})
\,\lbar{f}(\lbar{v}'_k)^{-1} \, P(\lbar{f}(\lbar{u}))^{-1} \bigg) \\
=&\ P\bigg( \lbar{f}(\lbar{u}) \, P(\lbar{f}(\lbar{u})) \, \lbar{f}(\lbar{v}_{1}) \, \lbar{f}(\lbar{v}_2 \cdots \lbar{v}_{k-1})
\,\lbar{f}((\lbar{v}'_k)^{-1}) \, P(\lbar{f}(\lbar{u}))^{-1} \bigg) \hspace{2cm} (\text{Eq.~(\ref{eq:rbdef2})} )\\
=&\ P\bigg( \lbar{f}(\lbar{u}) \, P(\lbar{f}(\lbar{u})) \, \lbar{f}(\lbar{v}_{1}) \, \lbar{f}(\lbar{v}_2 \cdots \lbar{v}_{k-1})
\,\lbar{f}(\lbar{v}_k) \, P(\lbar{f}(\lbar{u}))^{-1} \bigg) \hspace{3cm} (\lbar{v}_k = (\lbar{v}'_k)^{-1}  )\\
=&\ P\bigg( \lbar{f}(\lbar{u}) \, P(\lbar{f}(\lbar{u}))  \, \lbar{f}(\lbar{v}_1 \cdots \lbar{v}_{k})
\, P(\lbar{f}(\lbar{u}))^{-1} \bigg) \hspace{5cm}(\text{Eq.~(\ref{eq:rbdef3})} )\\
=&\ P\bigg( \lbar{f}(\lbar{u}) \, P(\lbar{f}(\lbar{u}))  \, \lbar{f}(\lbar{v})
\, P(\lbar{f}(\lbar{u}))^{-1} \bigg) \hspace{6cm}(\lbar{v} = \lbar{v}_1\cdots \lbar{v}_k)\\
=&\ P\Big(\lbar{f}(\lbar{u})\Big) \, P\Big(\lbar{f}(\lbar{v})\Big) \hspace{5cm} ((G, P) \text{ being a Rota-Baxter group})\\
=&\ \lbar{f}(\lc \lbar{u}\rc)\, \lbar{f}(\lc \lbar{v}\rc) \hspace{9cm}(\text{Eq.~(\ref{eq:rbdef1})} )\\
=&\ \lbar{f}(u)\, \lbar{f}(v).
\end{align*}

\noindent{\bf Subcase 1.2.} $\lbar{v}_k \notin \lc \rbgw\rc^{-1}$. Then
\begin{align*}
&\ \lbar{f}(u\diamond v)\\
=&\ \lbar{f}(\lc \lbar{u} \rc \diamond \lc \lbar{v}\rc)\\
=&\ \lbar{f}\Big( \Big\lc\lbar{u} \diamond \big( (u \diamond \lbar{v}) u^{-1}\big)\Big\rc \Big) \hspace{4cm} (\text{Eqs.~(\ref{eq:dia}) and~(\ref{eq:diam})})\\
=&\ P\Big( \lbar{f}\Big(\lbar{u} \diamond \big( (u \diamond \lbar{v}) u^{-1}\big)\Big) \Big) \hspace{4cm} (\text{Eq.~(\ref{eq:rbdef1})})\\
=&\ P\Big( \lbar{f}(\lbar{u}) \,\lbar{f}\big( (u \diamond \lbar{v}) u^{-1}\big) \Big) \hspace{4cm} (\text{induction on depth})\\
=&\ P\Big( \lbar{f}(\lbar{u}) \,\lbar{f}(u \diamond \lbar{v})\,\lbar{f}(u^{-1}) \Big) \hspace{4cm} (\text{Eq.~(\ref{eq:rbdef3})} ) \\
=&\ P\Big( \lbar{f}(\lbar{u}) \,\lbar{f}(u)\, \lbar{f}(\lbar{v})\,\lbar{f}(\lc \lbar{u}\rc^{-1}) \Big) \hspace{3.5cm} (\text{induction on depth and } u=\lc \lbar{u}\rc) \\
=&\ P\Big( \lbar{f}(\lbar{u}) \,\lbar{f}(\lc \lbar{u}\rc)\, \lbar{f}(\lbar{v})\,\lbar{f}(\lc \lbar{u}\rc)^{-1} \Big) \hspace{3.5cm} (\text{Eq.~(\ref{eq:rbdef2})}) \\
=&\ P\Big(\lbar{f}(\lbar{u})\, P\big(\lbar{f}(\lbar{u})\big)\,\lbar{f}(\lbar{v})\, P(\lbar{f}(\lbar{u}))^{-1}\Big) \hspace{3cm} (\text{Eq.~(\ref{eq:rbdef1})})\\
=&\ P\Big(\lbar{f}(\lbar{u})\Big)\, P\Big(\lbar{f}(\lbar{v})\Big) \hspace{5cm} ((G, P) \text{ being a Rota-Baxter group})\\
=&\ \lbar{f}(\lc \lbar{u}\rc)\lbar{f}(\lc \lbar{v}\rc) \hspace{6.3cm}(\text{Eq.~(\ref{eq:rbdef1})} )\\
=&\ \lbar{f}(u)\lbar{f}(v).
\end{align*}

\noindent{\bf Case 2.} Suppose $u= \lc \lbar{u}\rc^{-1}$ and $v=\lc \lbar{v}\rc^{-1}$ for some $\lbar{u},\lbar{v}\in \rbgw$. Then
\begin{align*}
 &\ \lbar{f}(u\diamond v)\\
 =&\ \lbar{f}(\lc \lbar{u}\rc^{-1} \diamond \lc \lbar{v}\rc^{-1}) \\
=&\ \lbar{f}\big((\lc \lbar{v}\rc\diamond \lc \lbar{u}\rc)^{-1}\big) \\
=&\ \Big(\lbar{f}(\lc \lbar{v} \rc\diamond \lc\lbar{u}\rc)\Big)^{-1} \hspace{2cm} (\text{Eq.~(\ref{eq:rbdef2})})\\
=&\ \Big(\lbar{f}(\lc \lbar{v} \rc) \,\lbar{f}(\lc\lbar{u}\rc)\Big)^{-1} \hspace{2cm} (\text{Case 1})\\
=&\  \big(\lbar{f}(\lc \lbar{u} \rc)\big)^{-1} \,\big(\lbar{f}(\lc\lbar{v}\rc)\big)^{-1}\\
=&\ \lbar{f}(\lc \lbar{u} \rc^{-1}) \,\lbar{f}(\lc\lbar{v}\rc^{-1}) \hspace{2cm} (\text{Eq.~(\ref{eq:rbdef2})}) \\
=&\ \lbar{f}(u) \lbar{f}(v).
\end{align*}

\noindent{\bf Case 3.} Suppose $\bre(u) = \bre(v)=1$, but $u, v$ are not in the previous two cases. Then by definition, $u\diamond v= uv$ is the concatenation. It follows from Eq.~(\mref{eq:rbdef3}) that
\begin{align*}
\lbar{f}(u\diamond v)=&\ \lbar{f}(uv)= \lbar{f}(u)\lbar{f}(v).
\end{align*}

\noindent{\bf Case 4.} Suppose at least one of $\bre(u)$ or $\bre(v)$ is greater than $1$. Let $u=u_{1}\cdots u_{m}$ and $v=v_{1}\cdots v_{n}$ be in the standard forms.
Then
\begin{align*}
&\ \lbar{f}(u\diamond v)\\
=&\ \lbar{f}\big(u_{1}\cdots u_{m-1}(u_{m}\diamond v_{1})v_{2}\cdots v_{n}\big) \\
=&\ \lbar{f}(u_{1})\cdots \lbar{f}(u_{m-1})\, \lbar{f}(u_{m}\diamond v_{1}) \,\lbar{f}(v_{2}) \cdots \lbar{f}(v_{n}) \hspace{2cm} (\text{Eq.~(\ref{eq:rbdef3})})\\
=&\ \lbar{f}(u_{1})\cdots \lbar{f}(u_{m})\, \lbar{f}(v_{1})\cdots \lbar{f}(v_{n}) \hspace{4cm} (\text{Cases 1, 2, 3})\\
=&\ \lbar{f}(u_{1}\cdots  u_{m}) \, \lbar{f}(v_{1}\cdots v_{n}) \hspace{5cm}(\text{Eq.~(\ref{eq:rbdef3})})\\
=&\ \lbar{f}(u)\lbar{f}(v).
\end{align*}

In summary, the map $\lbar{f}$ is a Rota-Baxter group homomorphism.

({\bf Uniqueness}) Notice that Eqs.~(\ref{eq:bfdep0})--(\ref{eq:rbdef3}) give the only possible way to define $\lbar{f}(w)$
in order for $\lbar{f}$ to be a Rota-Baxter group homomorphism that extends $f$. This proves the uniqueness.

Now the proof of Theorem~\mref{thm:rtgp2} is completed.
\end{proof}

\noindent
{\bf Acknowledgments.}  This work is supported by
NSFC (12071191, 1861051 and 12171211) and the Natural Science Foundation of Jiangxi Province (20192ACB21008).  


\noindent
{\bf Competing Interests.} On behalf of all authors, the corresponding author states that there is no conflict of interest.

\noindent
{\bf Data Availability.} The manuscript has no associated data.

\end{document}